\documentclass[a4paper,12pt]{amsart}

\usepackage{mathtools}
\usepackage{amssymb}
\usepackage{amsthm}
\usepackage{bm}
\usepackage{mathrsfs}
\usepackage{mathcomp}

\usepackage{enumitem}

\newlist{enumarabic}{enumerate}{5}
\setlist[enumarabic]{label=\arabic*}
\newlist{enumarabicd}{enumerate}{5}
\setlist[enumarabicd]{label=\arabic*.}
\newlist{enumarabicp}{enumerate}{5}
\setlist[enumarabicp]{label=(\arabic*)}

\newlist{enumAlph}{enumerate}{5}
\setlist[enumAlph]{label=\Alph*}
\newlist{enumAlphd}{enumerate}{5}
\setlist[enumAlphd]{label=\Alph*.}
\newlist{enumAlphp}{enumerate}{5}
\setlist[enumAlphp]{label=(\Alph*)}

\newlist{enumalph}{enumerate}{5}
\setlist[enumalph]{label=\alph*}
\newlist{enumalphd}{enumerate}{5}
\setlist[enumalphd]{label=\alph*.}
\newlist{enumalphp}{enumerate}{5}
\setlist[enumalphp]{label=(\alph*)}

\newlist{enumRoman}{enumerate}{5}
\setlist[enumRoman]{label=\Roman*}
\newlist{enumRomand}{enumerate}{5}
\setlist[enumRomand]{label=\Roman*.}
\newlist{enumRomanp}{enumerate}{5}
\setlist[enumRomanp]{label=(\Roman*)}

\newlist{enumroman}{enumerate}{5}
\setlist[enumroman]{label=\roman*}
\newlist{enumromand}{enumerate}{5}
\setlist[enumromand]{label=\roman*.}
\newlist{enumromanp}{enumerate}{5}
\setlist[enumromanp]{label=(\roman*)}

\usepackage{graphicx}
\usepackage[dvipsnames,svgnames,table]{xcolor}

\usepackage{float}
\usepackage{array}
\usepackage{booktabs}
\usepackage{tabularx}
\usepackage{longtable}
\usepackage{multirow}

\usepackage{tikz}
\usetikzlibrary{arrows}
\usetikzlibrary{calc}
\usetikzlibrary{intersections}
\usetikzlibrary{patterns}
\usetikzlibrary{quotes}
\usetikzlibrary{shapes}
\usetikzlibrary{through}
\usepackage{tikz-cd}
\usepackage{mdframed}

\usepackage[
  bookmarks=true,%
  bookmarksnumbered=true,%
  setpagesize=false,%
  colorlinks=true,%
  linkcolor=Blue,%
  citecolor=YellowOrange,%
  urlcolor=HotPink,%
]{hyperref}

\usepackage{cleveref}
\theoremstyle{plain}

\newtheorem{theorem}{Theorem}[section]
\newtheorem{proposition}[theorem]{Proposition}
\newtheorem{corollary}[theorem]{Corollary}
\newtheorem{lemma}[theorem]{Lemma}

\theoremstyle{definition}

\newtheorem{definition}[theorem]{Definition}

\theoremstyle{remark}

\newtheorem{remark}[theorem]{Remark}
\newtheorem{example}[theorem]{Example}

\newtheorem*{acknowledgement}{Acknowledgement}

\crefname{theorem}{Theorem}{Theorem}
\Crefname{theorem}{Theorem}{Theorem}
\crefname{proposition}{Proposition}{Proposition}
\Crefname{proposition}{Proposition}{Proposition}
\crefname{corollary}{Corollary}{Corollary}
\Crefname{corollary}{Corollary}{Corollary}
\crefname{lemma}{Lemma}{Lemma}
\Crefname{lemma}{Lemma}{Lemma}
\crefname{claim}{Claim}{Claim}
\Crefname{claim}{Claim}{Claim}
\crefname{fact}{Fact}{Fact}
\Crefname{fact}{Fact}{Fact}
\crefname{definition}{Definition}{Definition}
\Crefname{definition}{Definition}{Definition}
\crefname{notation}{Notation}{Notation}
\Crefname{notation}{Notation}{Notation}
\crefname{remark}{Remark}{Remark}
\Crefname{remark}{Remark}{Remark}
\crefname{example}{Example}{Example}
\Crefname{example}{Example}{Example}
\crefname{problem}{Problem}{Problem}
\Crefname{problem}{Problem}{Problem}
\crefname{answer}{Answer}{Answer}
\Crefname{answer}{Answer}{Answer}
\crefname{section}{Section}{Section}
\Crefname{section}{Section}{Section}
\crefname{assumption}{Assumption}{Assumption}
\Crefname{assumption}{Assumption}{Assumption}

\newcommand{\parenlr}[1]{\left(#1\right)}

\newcommand{\abracket}[1]{\langle#1\rangle}

\newcommand{\abs}[1]{\lvert#1\rvert}

\newcommand{\norm}[1]{\lVert#1\rVert}

\newcommand{\set}[2]{\{#1\mid#2\mbox{}\}}

\newcommand{\map}[3]{#1\colon#2\to#3}
\newcommand{\restr}[2]{#1|_{#2}}

\newcommand{\id}{\mathrm{id}}

\DeclareMathOperator{\End}{End}
\DeclareMathOperator{\Hom}{Hom}
\DeclareMathOperator{\Image}{Im} % \Im is already defined

\DeclareMathOperator{\SW}{SW}

\newcommand{\mbfS}{\mathbf{S}}

\newcommand{\mbfh}{\mathbf{h}}

\newcommand{\mcalD}{\mathcal{D}}
\newcommand{\mcalE}{\mathcal{E}}
\newcommand{\mcalF}{\mathcal{F}}

\newcommand{\mcalH}{\mathcal{H}}

\newcommand{\mcalM}{\mathcal{M}}

\newcommand{\mcalO}{\mathcal{O}}

\newcommand{\mcalU}{\mathcal{U}}
\newcommand{\mcalV}{\mathcal{V}}
\newcommand{\mcalW}{\mathcal{W}}

\newcommand{\mscrF}{\mathscr{F}}

\newcommand{\mscrV}{\mathscr{V}}
\newcommand{\mscrW}{\mathscr{W}}

\newcommand{\mfraks}{\mathfrak{s}}
\newcommand{\mfrakt}{\mathfrak{t}}
\newcommand{\mfraku}{\mathfrak{u}}

\newcommand{\mbbB}{\mathbb{B}}

\newcommand{\mbbX}{\mathbb{X}}

\newcommand{\Z}{\mathbb{Z}}

\newcommand{\R}{\mathbb{R}}
\newcommand{\C}{\mathbb{C}}
\newcommand{\HB}{\mathbb{H}}

\usepackage{empheq, comment}
\numberwithin{equation}{section}

\newcommand{\hplus}[1]{\mcalH^+(#1)}
\newcommand{\shplus}[1]{S(\hplus{#1})}

\DeclareMathOperator{\Int}{Int}

\DeclareMathOperator{\Harm}{Harm}
\DeclareMathOperator{\Diff}{Diff}

\DeclareMathOperator{\ind}{ind}

\DeclareMathOperator{\Fr}{Fr}

\DeclareMathOperator{\univ}{univ}
\DeclareMathOperator{\pt}{pt}

\DeclareMathOperator{\Alb}{Alb}
\DeclareMathOperator{\Pic}{Pic}

\newcommand{\Diffplus}{\Diff^+(X)}
\newcommand{\Diffspin}{\Diff^+(X, \mfraks)}
\newcommand{\Diffspiniso}{\Diff^+(X, [\mfraks])}

\newcommand{\Xunivspiniso}{\mbbX^{[\mfraks]}_{\univ}}

\newcommand{\tiota}{\tilde\iota}
\newcommand{\tf}{\tilde f}

\newcommand{\tmcalE}{\tilde\mcalE}

\newcommand{\pprime}{{\prime\prime}}

\begin{document}

\title{A gerbe-like construction in gauge theory II: the case of homology tori}
\author{Mitsuyoshi Adachi}

\begin{abstract}
  In the previous paper, the author showed that for a smooth family $X \to \mathbb{X} \to B$ of a homotopy $K3$ surface, the obstruction for the tangent bundle along the fibers $T_B \mathbb{X}$ to have a spin structure is canonically isomorphic to the obstruction for $\mathcal{H}^+(\mathbb{X})$, the vector bundle over $B$ consisting of self-dual harmonic 2-forms, to have a spin structure.

  In this paper, we show an analogous result for homology tori with odd determinant. The strategy for proof is similar to the case of homotopy $K3$ surfaces: take the determinant line bundle of the $K$-theoretic Seiberg--Witten invariant and construct an anti-linear $\mathbb{Z}/4$-action on it at the representative level.

  We also see that the anti-linear $\mathbb{Z}/4$-action possesses the information of the ordinary mod 2 Seiberg--Witten invariant. This recovers part of the result by Baraglia(2023) which computes the mod 2 Seiberg--Witten invariants for any closed spin 4-manifold.
\end{abstract}

\maketitle

\section{Introduction}

Gauge theory is a powerful tool to study the differential topology of 4-manifolds. In recent years it has been used to study families of 4-manifolds, and reveals properties of the diffeomorphism groups of 4-manifolds. The study of the diffeomorphism groups of 4-manifolds through gauge theory was initiated by Ruberman \cite{Ruberman-an-obstruction1998,Ruberman-polynomial-invarint-of-diffeo-1999}.

In the previous paper, the author has proved the following result. Let $X \to \mbbX \to B$ be an oriented smooth family of a homotopy $K3$ surface. Then there is a canonical isomorphism between (1) the $O(1)$-gerbe which represents the obstruction for $T_B \mbbX$ to admit a spin structure and (2) the $O(1)$-gerbe for $\hplus{\mbbX}$, the vector bundle whose fiber consists of self-dual harmonic 2-forms, to admit a spin structure. See \cite[Remark 6.1]{Adachi-gerbe-like-2026} for details. This  strengthens the theorem by Baraglia--Konno \cite{Baraglia--Konno-2022} which states that $w_2(\hplus{\mbbX})$ vanishes if $T_B \mbbX$ admits a spin structure. The purpose of this paper is to prove an analogous result in the case of homology 4-tori with odd determinant.

A homology 4-torus is an oriented closed 4-manifold with the same homology as that of the 4-torus. For a homology torus $X$, an invariant called the determinant is defined by
\[
  \det X = \abs{\abracket{\alpha_1 \cup \alpha_2 \cup \alpha_3 \cup \alpha_4}, [X]}.
\]
Here $\alpha_i$ ($i = 1, 2, 3, 4$) is a basis of $H^1(X; \Z)$, and $[X]$ is the fundamental class. When we consider a homology 4-torus $X$, the situation is different from that of the homotopy $K3$ surfaces since the first Betti number is nonzero. There are two main differences. One is that the diffeomorphisms which preserve the orientation do not necessarily preserve the isomorphism class of spin structures of $X$. Therefore, the obstruction for $X \to \mbbX \to B$ to admit a spin structure along the fibers can be considered in three stages.
\begin{enumarabicp}
  \item Fix a spin structure $\mfraks$ on $X$.
  \item Reduce the structure group of $\mbbX$ to the topological group consisting of diffeomorphisms which preserve the isomorphism class of $\mfraks$.
  \item Extend $\mfraks$ to a spin structure on $T_B \mbbX$.
\end{enumarabicp}
In this paper, we fix a spin structure $\mfraks$ on $X$ and assume that the obstruction of (2) already vanishes. Our task is to consider the relationship between the $O(1)$-gerbe appearing as the obstruction of (3) and the $O(1)$-gerbe as the obstruction for $\hplus{\mbbX}$ to admit a spin structure.

Another difference is that the isomorphism of $O(1)$-gerbes is formulated not over $B$, but over $\Alb(\mathcal{X})$, which is a certain torus bundle over $B$. For the definition of $\Alb(\mathcal{X})$, see \cref{def:Alb(X)}. The immediate reason for this is technical: it arises from the use of a canonical principal $U(1)$-bundle over the Seiberg-Witten moduli space in the proof. Whether the isomorphism of $O(1)$-gerbes holds over $B$ remains unknown to the author. However, if $X \to \mathcal{X} \to B$ admits a global section, it induces a global section of $\Alb(\mathcal{X}) \to B$. Therefore, given this auxiliary data, the $O(1)$-gerbes become canonically isomorphic over $B$. We will later discuss the non-triviality of the smooth isotopy class of the boundary Dehn twist for homology 4-tori with odd determinant. For this application, it suffices to consider only the case where $\mathcal{X}$ admits a global section.

\subsection{Main results}

As previously mentioned, the situation for homology 4-tori is more complex than that for homotopy $K3$ surfaces due to the non-vanishing first homology. We begin by organizing this with the introduction of some notation. Let $X$ be a homology torus with odd determinant, and let $\mathfrak{s}$ be a spin structure on it. In this paper, we formulate a spin structure in a way that does not depend on a Riemannian metric on $X$: specifically, a spin structure is a lift of the principal $GL^+(4, \mathbb{R})$-bundle $TX$ to a principal $\widetilde{GL}^+(4, \mathbb{R})$-bundle. Let $\mathrm{Diff}^+(X)$ be the topological group of orientation-preserving diffeomorphisms of $X$, and define
\begin{align*}
  \Diffspin & = \set{(f, \tf)}{f \in \Diffplus,                                                                                          \\
            & \text{$\map{\tf}{\mfraks}{\mfraks}$ is a $\widetilde{GL}^+(4, \R)$-equivariant lift of $\map{df}{\Fr^+(TX)}{\Fr^+(TX)}$}}.
\end{align*}
Since a spin structure on $X$ is not necessarily unique, the natural map
\[
  \Diffspin \to \Diffplus
\]
is not always surjective. Let $\Diffspiniso$ denote the image of this map. We then obtain the exact sequence
\[
  1 \to \mathbb{Z}/2 \to \Diffspin \to \Diffspiniso \to 1.
\]
The main theorem of this paper is as follows.

\begin{theorem}\label{thm:Diffspin times_Z/2 Spin(3)}
  Let $X$ be a homology 4-torus with odd determinant and $E\Diffspiniso \to B\Diffspiniso$ be the universal principal $\Diffspiniso$-bundle. Let $\Xunivspiniso$ be the $X$-bundle associated with $E\Diffspiniso$. Fix a continuous family of Riemannian metrics on $\Xunivspiniso$. Then there is a canonical principal $\Diffspin \times_{\Z/2} Spin(3)$-bundle over $\Alb(\Xunivspiniso)$ that lifts the principal $\Diffspiniso \times SO(3)$-bundle
  \[
    \pi_{\Alb(\Xunivspiniso)}^\ast(E\Diffspiniso \times_{B\Diffspiniso} \Fr^{SO}(\hplus{\Xunivspiniso})),
  \]
  where $\map{\pi_{\Alb(\Xunivspiniso)}}{\Alb(\Xunivspiniso)}{B\Diffspiniso}$ is the natural projection.
\end{theorem}

For the definition of $\Alb(\Xunivspiniso)$, see \cref{def:Alb(X)}. The above theorem gives an equality on characteristic classes. Let
\[
  \alpha(\Xunivspiniso, \mfraks) \in H^2(B\Diffspiniso; \Z/2)
\]
be the primary obstruction for $E\Diffspiniso$ to have a lift to a principal $\Diffspin$-bundle. The vanishing of $\alpha(\Xunivspiniso, \mfraks)$ is equivalent to $E\Diffspiniso$ having a lift to a principal $\Diffspin$-bundle. We also have a characteristic class
\[
  w_2(\hplus{\Xunivspiniso}) \in H^2(B\Diffspiniso; \Z/2).
\]
Since the isomorphism class of $\hplus{\Xunivspiniso}$ is independent of the choice of the family of the Riemannian metrics on $\Xunivspiniso$, $w_2(\hplus{\Xunivspiniso})$ is also independent of it.

\begin{corollary}\label{cor:alpha = w_2}
  Let $X$ be a homology 4-torus with odd determinant. Fix a continuous family of Riemannian metrics on $\Xunivspiniso$. Then we have
  \[
    \pi_{\Alb(\Xunivspiniso)}^\ast \alpha(\Xunivspiniso, \mfraks) = \pi_{\Alb(\Xunivspiniso)}^\ast w_2(\hplus{\Xunivspiniso}) \in H^2(\Alb(\Xunivspiniso); \Z/2).
  \]
\end{corollary}

\cref{thm:Diffspin times_Z/2 Spin(3)} is the immediate corollary of the following theorem.

\begin{theorem}\label{thm:Spin(3) from Diffspin}
  Let $X$ be a homology 4-torus with odd determinant and let $\mcalE \to B$ be a principal $\Diffspiniso$-bundle over a contractible paracompact Hausdorff space. Fix a continuous family of Riemannian metrics on $\mbbX = \mcalE \times_{\Diffspiniso} X$. Assume that a lift $\tilde{\mcalE}$ of $\mcalE$ to a principal $\Diffspin$-bundle is given. Then a spin structure $\mfrakt$ on $\pi_{\Alb(\mbbX)}^\ast \hplus{\mbbX}$ is canonically constructed, where $\map{\pi_{\Alb(\mbbX)}}{\Alb(\mbbX)}{B}$ is the natural projection. Furthermore, the correspondence between $\tilde \mcalE$ and $\mfrakt$ is functorial. Under this functoriality, the isomorphisms
  \[
    \map{+1, -1}{\tilde \mcalE}{\tilde \mcalE}
  \]
  corresponds to
  \[
    \map{+1, -1}{\mfrakt}{\mfrakt}
  \]
  respectively.
\end{theorem}

The meaning of the functoriality is as follows. Let $\mcalE \to B$, $\tilde \mcalE$ be as in \cref{thm:Spin(3) from Diffspin} and $\mcalE^\prime \to B^\prime$, $\tilde \mcalE^\prime$ be another pair of such data. Fix continuous families of Riemannian metrics on $\mbbX$ and $\mbbX^\prime = \mcalE^\prime \times_{\Diffspiniso} X$. Let $\mfrakt$, $\mfrakt^\prime$ be the spin structures on $\pi_{\Alb(\mbbX)}^\ast \hplus{\mbbX}$ and $\pi_{\Alb(\mbbX)}^{\prime \ast} \hplus{\mbbX^\prime}$, where $\map{\pi_{\Alb(\mbbX)}^\prime}{\Alb(\mbbX^\prime)}{B^\prime}$ is the natural projection. Assume that we have continuous maps
\[
  \map{f}{B^\prime}{B},\ \map{\tilde f}{\tilde \mcalE^\prime}{\tilde \mcalE},
\]
where $\tilde f$ is a $\Diffspin$-equivariant map covering $f$, and the induced map from $\mbbX^\prime$ to $\mbbX$ preserves the families of Riemannian metrics. Then we can construct a canonical $Spin(3)$-equivariant map
\[
  \map{f_\ast}{\mfrakt^\prime}{\mfrakt}
\]
which covers the map from $\pi_{\Alb(\mbbX)}^\ast \hplus{\mbbX^\prime}$ to $\hplus{\mbbX}$ induced by $\tilde f$.

The isomorphisms $\map{\pm 1}{\tilde \mcalE}{\tilde \mcalE}$ come from $\pm 1 \in \Diffspin$ and $\map{\pm 1}{\mfrakt}{\mfrakt}$ come from $\pm 1 \in Spin(3)$: thus they cover the identity map on $\mcalE$ or $\pi_{\Alb(\mbbX)}^\ast \hplus{\mbbX}$ respectively.

As an application of \cref{cor:alpha = w_2}, we give an alternative proof of the following result by Qiu \cite{Qiu-Dehn2025-arXiv}.

\begin{theorem}\label{thm:bd Dehn twist}
  Let $X$ be a homology 4-torus with odd determinant. Then the boundary Dehn twist of $X^\prime = X \setminus \Int D^4$ is not isotopic to the identity relative to $\partial X^\prime$.
\end{theorem}

The original result of Qiu \cite{Qiu-Dehn2025-arXiv} is that the Dehn twist on the neck of the connected sum of two homology 4-tori with odd determinant is not smoothly isotopic to the identity. His strategy is based on the idea of Kronheimer--Mrowka \cite{Kronheimer--Mrowka-Dehn-twist-K3-2020} in the case of homotopy $K3$ surfaces. He computes the $Pin(2)$-equivariant Bauer--Furuta invariant of the mapping torus of the Dehn twist. On the other hand, for the case of the boundary Dehn twist of a homotopy $K3$ surface, Kronheimer--Mrowka \cite{Kronheimer--Mrowka-Dehn-twist-K3-2020}  presented a different approach. This is based on a result due to Baraglia--Konno \cite{Baraglia--Konno-2022} that if $w_2(\hplus{\mbbX}) = 0$, then the tangent bundle along the fibers of $\mbbX$ admits a spin structure. Our approach to \cref{thm:bd Dehn twist} is based on the latter argument.

\subsection{Outline of the proof}

The main part of the proof is to construct a spin structure on $\hplus{\mbbX}$, given that $X \to \mbbX \to B$ is equipped with a lift of the structure group to $\Diffspin$. The essence of the proof is the same as that in \cite{Adachi-gerbe-like-2026}: roughly speaking, take the $K$-theoretic degree of the finite-dimensional approximation of the Seiberg--Witten map modded out by $U(1)$ at the representative level, and then take its determinant line bundle. As discussed in \cite[Section 3]{Adachi-gerbe-like-2026}, such a complex line bundle over $\shplus{\mbbX}$ with an anti-linear isomorphism on it gives a spin structure on $\hplus{\mbbX}$. (Here we use $b^+(X) = 3$.) The anti-linear isomorphism is derived from the $j$-action coming from the $Pin(2)$-equivariance of the Seiberg--Witten map. Then we study how the self-isomorphism of $\mbbX$ acts on the spin structure on $\hplus{\mbbX}$. This analysis gives the proof of \cref{thm:Spin(3) from Diffspin}. The construction should be ``at the representative level'', not as an element of the $K$-group, in order to discuss canonical construction.

One main difference from the case of homotopy $K3$ surface is that we have to consider how to equip $H^1(\mbbX; \R)$, the vector bundle whose fiber at $b$ is $H^1(\mbbX_b; \R)$, with a spin structure. This is necessary since we want to take a $K$-theoretic degree. Actually, there is no canonical way to give it a spin structure. To remedy this, we arbitrarily choose a spin structure on it and prove that the resulting spin structure on $\hplus{\mbbX}$ is independent of the choice. More precisely, if we have another choice of the spin structure on $H^1(\mbbX; \R)$, then we have a canonical isomorphism between the spin structures on $\hplus{\mbbX}$.

As a byproduct of the construction described above, we can determine the mod 2 Seiberg--Witten invariants of closed spin 4-manifolds with $b^+ \geq 3$. We discuss it in \cref{sec:mod 2 SW}. This recovers part of the result by Baraglia \cite{baraglia-mod-2023-arxiv}, in which the mod 2 Seiberg--Witten invariants of closed spin 4-manifolds is completely determined even in the case of $b^+ = 1, 2$. Prior to his work, there are contributions by Morgan--Szab\'{o} \cite{morgan1996seiberg} in the case of homotopy $K3$ surfaces, Ruberman--Strle \cite{Ruberman--Strle-mod2-SW-homology-tori-2000} in the case of homology tori, and Bauer \cite{Bauer-almost-complex-2008} and Li \cite{Li-quaternionic-bundles-2006} in other special cases. Our method here is summarized as follows. The existence of the complex line bundle on $\shplus{\mbbX}$ with some anti-linear isomorphism gives a topological constraint expressed as an equality. (See \cref{prop:top restr for anti-lin map}.) One side of the equality can be computed using the mod 2 Seiberg--Witten invariant, and this leads to the determination of it.

It is natural to ask whether analogous statements to \cref{thm:Diffspin times_Z/2 Spin(3)} hold for closed spin 4-manifolds other than homotopy $K3$ surfaces and homology 4-tori with odd determinant. At least, if we use the above technique (taking the determinant line bundle of the $K$-theoretic degree of the finite-dimensional approximation of the Seiberg--Witten map deformed and modded out by $U(1$)), there are no other examples. (Strictly speaking, the necessary assumption is the restriction for the Betti numbers.) The reason is as follows. First, a spin structure can be described using an appropriate complex vector bundle over a sphere bundle and an anti-linear isomorphism only in the case of rank 3 vector bundles. This forces us to consider the case $b^+ = 3$. Furthermore, the success of our construction necessitates the mod 2 Seiberg--Witten invariant being equal to 1. This imposes constraints on the Betti numbers.

\subsection{Organization of the paper}

In \cref{sec:FDA}, we describe the finite-dimensional approximations of the Seiberg--Witten map for the closed spin 4-manifold with possibly $b_1 > 0$. In \cref{sec:const of spin str}, we construct a spin structure on $\hplus{\mbbX}$, given that $X \to \mbbX \to B$ is equipped with a lift of the structure group to $\Diffspin$. In \cref{sec:proof of the main theorems}, we prove the main theorems. In \cref{sec:mod 2 SW}, the mod 2 Seiberg--Witten invariant is computed in the case of a closed spin 4-manifold with $b^+ \geq 3$.

\begin{acknowledgement}
  The author would like to express his deep gratitude to Mikio Furuta for helpful suggestions and continuous encouragements. The author also thanks Nobuhiro Nakamura for encouraging the author to think the homology tori case. The author would also like to express his appreciation to Hokuto Konno for helpful comments and continuous encouragements.
\end{acknowledgement}

\section{Finite-dimensional approximations for $b_1 > 0$}\label{sec:FDA}

In this section, we describe the finite-dimensional approximations of the families Seiberg--Witten map for $b_1 > 0$. The basic idea is the same as \cite[Section 4.1]{Adachi-gerbe-like-2026}, but the argument has to be modified since we need one additional data represented by the element of $\Alb(\mbbX)$ if the 4-manifold has positive first Betti number. For a reference, see Bauer--Furuta \cite{Bauer--Furuta-2004} for example.

After mentioning the construction of the finite-dimensional approximations, we abstract the property of them essential for our purpose.

\subsection{The Seiberg--Witten map for $b_1 > 0$}\label{subsec:SW map B_1 > 0}

The goal of this subsection is to construct a finite-dimensional approximation of the families Seiberg--Witten map. Let $X$ be a closed oriented 4-manifold and $\mfraks$ be a spin structure on $X$. In this subsection, $X$ need not be the homology 4-torus with odd determinant. Let $\mcalE \to B$ be a principal $\Diffspiniso$-bundle over a paracompact Hausdorff space $B$. Later we assume that $B$ is also contractible. Fix a continuous family of Riemannian metrics on $\mbbX = \mcalE \times_{\Diffspiniso} X$, and assume that a lift $\tilde \mcalE$ of $\mcalE$ to a principal $\Diffspin$-bundle is given. Let
\begin{align*}
  \mscrV & = \coprod_{b \in B} \Omega^1(\mbbX_b; i\R) \times \Gamma(S^+_b),                                 \\
  \mscrW & = \coprod_{b \in B} \Omega^0_0(\mbbX_b; i\R) \times \Omega^+(\mbbX_b; i\R) \times \Gamma(S^-_b),
\end{align*}
where $S^\pm_b$ are spinor bundles and $\Omega^0_0(\mbbX_b; i\R)$ is the $L^2$-complement of $H^0(\mbbX_b; i\R)$. For $(a, \phi) \in \mscrV_b$, let
\[
  D(a, \phi) = (d^\ast a, d^+ a, D^+ \phi),
\]
where $D^+$ denotes the Dirac operator determined by the canonical reference connection. Also, define
\[
  Q(a, \phi) = (0, -q(\phi), a \cdot \phi),
\]
where $q(\phi) \in \Omega^+(\mbbX_b; i\R)$ is defined by
\[
  q(\phi) \cdot \psi = \abracket{\psi, \phi} \phi - \frac{1}{2} \abs{\phi}^2 \psi.
\]
The families Seiberg--Witten map is
\[
  \map{D + Q}{\mscrV}{\mscrW}.
\]
As in \cite{Adachi-gerbe-like-2026}, we introduce the family of perturbations of the families Seiberg--Witten map in order to avoid reducible solutions. Let $\shplus{\mbbX}$ be the sphere bundle of $\hplus{\mbbX}$. Then
\[
  \shplus{\mbbX} \times_B \mscrV,\ \shplus{\mbbX} \times_B \mscrW
\]
are Hilbert bundles over $\shplus{\mbbX}$, and $D$, $Q$ lift to maps between the above bundles, which we use the same notation. Let
\[
  \map{\Delta}{\shplus{\mbbX} \times_B \mscrV}{\shplus{\mbbX} \times_B \mscrW}
\]
be the lift of the tautological section of $\shplus{\mbbX} \times \hplus{\mbbX} \to \shplus{\mbbX}$. We call
\[
  \map{F = D + Q - \Delta}{\shplus{\mbbX} \times_B \mscrV}{\shplus{\mbbX} \times_B \mscrW}
\]
the perturbed families Seiberg--Witten map. Note that this map is $Pin(2)$-equivariant.

A gauge transformation $\map{f}{\mbbX_b}{U(1)}$ acts on $\mscrV$ and $\mscrW$ by
\[
  f \cdot (a, \phi) = (a - f^{-1} df, f \phi),\ f \cdot (g, \omega, \psi) = (g, \omega, f \psi).
\]
For each $b \in B$, the topological group $\Harm(\mbbX_b, U(1))$ consisting of harmonic functions is isomorphic to
\[
  H^1(\mbbX_b; 2 \pi i \Z) \times U(1),
\]
and so we cannot obtain finite-dimensional approximations of the (perturbed) families Seiberg--Witten map unless we mod out the map by the $H^1(\mbbX_b; 2 \pi i \Z)$ component. However, the isomorphism between $\Harm(\mbbX_b, U(1))$ and $H^1(\mbbX_b; 2 \pi i \Z) \times U(1)$ is not unique, and therefore we need additional data to fix the isomorphism. We introduce what we will call ``the Albanese torus without base points''\footnote{The author learned the construction from Mikio Furuta.}.

\begin{definition}\label{def:Alb(X)}
  Let $X$ be a closed oriented 4-manifold and fix a Riemannian metric. Identify $H^1(X; \Z)$ with the vector space consisting of harmonic 1-forms with integral period. Then we define $\Alb(X)$ by
  \[
    \Alb(X) = (X \times \Hom(H^1(X; \Z), U(1)))/ \sim,
  \]
  where $(x, f) \sim (y, g)$ if and only if for any $h \in \mcalH^1(X; \Z)$ we have
  \[
    g(h) = \exp\parenlr{2 \pi i \int_\gamma h} f(h),
  \]
  where $\gamma$ is a smooth path from $x$ to $y$. Note that the above expression does not depend on the choice of $\gamma$ since the integrations differ by integers for different choices of $\gamma$.
\end{definition}

It is straightforward that $\Alb(X)$ is isomorphic to $\Hom(H^1(X; \Z), U(1))$ but not canonically. We have the following proposition.

\begin{proposition}\label{prop:Alb}
  Let $X$ be a closed oriented 4-manifold and fix a Riemannian metric. Then $\Alb(X)$ is canonically isomorphic to the set of all homomorphisms
  \[
    \map{\Phi}{H^1(X; 2 \pi i \Z)}{\Harm(X, U(1))}
  \]
  satisfying
  \[
    \varphi^{-1} d\varphi = a
  \]
  for each $a \in H^1(X; 2 \pi i \Z)$ and $\varphi = \Phi(a)$.
\end{proposition}

\begin{proof}
  First fix an element $\bm{f} = [x, f] \in \Alb(X)$. Then we define a homomorphism
  \[
    \map{\Phi}{H^1(X; 2 \pi i \Z)}{\Harm(X, U(1))}
  \]
  by
  \[
    \Phi(a)(y) = f(a / 2 \pi i) \exp\parenlr{\int_\gamma a_{\text{harm}}},
  \]
  where $a_{\text{harm}}$ is a harmonic 1-form representing $a$ and $\gamma$ is a path from $x$ to $y$. This map is independent of the choice of $(x, f)$. The inverse correspondence is given as follows. For a homomorphism
  \[
    \map{\Phi}{H^1(X; 2 \pi i \Z)}{\Harm(X, U(1))}
  \]
  and $x \in X$, we define
  \[
    \map{f}{H^1(X; \Z)}{U(1)}
  \]
  by
  \[
    f(h) = \Phi(2 \pi i h)(x).
  \]
  Then $[x, f]$ does not depend on the choice of $x$.
\end{proof}

By applying the above construction to $\mbbX$, we obtain a torus bundle
\[
  \Alb(\mbbX) = \coprod_{b \in B} \Alb(\mbbX_b) \to B.
\]
By pulling back the perturbed Seiberg--Witten map to $\shplus{\mbbX} \times \Alb(\mbbX)$, we can canonically mod out the map by the $H^1(\mbbX_b; 2 \pi i\Z)$ component. As a result, we obtain a map
\[
  \map{F^\prime}{\shplus{\mbbX} \times_B \mscrV^\prime}{\shplus{\mbbX} \times_B \mscrW^\prime},
\]
where
\begin{align*}
  \mscrV^\prime & = \coprod_{(b, \bm{f}, a) \in \mbbB} \parenlr{H^1(\mbbX_b; i\R) \times \Omega^1_0(\mbbX_b; i\R) \times \Gamma(S^+_b)} / H^1(\mbbX_b; 2 \pi i\Z),                               \\
  \mscrW^\prime & = \coprod_{(b, \bm{f} , a) \in \mbbB} \parenlr{H^1(\mbbX_b; i\R) \times \Omega^0_0(\mbbX_b; i\R) \times \Omega^+(\mbbX_b; i\R) \times \Gamma(S^-_b)} / H^1(\mbbX_b; 2 \pi i\Z)
\end{align*}
are $Pin(2)$-equivariant Hilbert bundles over $\mbbB = \Alb(\mbbX) \times_B \Pic^{\mfraks}(\mbbX)$. Here $\Omega^1_0(\mbbX_b; i\R)$ is the $L^2$-orthogonal complement of $H^1(\mbbX_b; i\R)$ in $\Omega^1(\mbbX_b; i\R)$ and
\[
  \Pic^\mfraks(\mbbX_b) = H^1(\mbbX_b; i\R) / H^1(\mbbX_b; 2 \pi i\Z).
\]

\subsection{Finite-dimensional approximation}

In this subsection we construct a finite-dimensional approximation of $F^\prime$. We will use the same notation as in \cref{subsec:SW map B_1 > 0}. The formalization used here is the mixture of the methods by Bauer--Furuta \cite{Bauer--Furuta-2004} and Furuta \cite{Furuta-11/8-inequality2001}.

First, we consider the Sobolev completion of $\mscrV^\prime$ and $\mscrW^\prime$. Fix an integer $k \geq 3$. For $v \in \mscrV^\prime$ and $w \in \mscrW^\prime$, let
\[
  \norm{v}_{\mscrV^\prime}^2 = \norm{(D^\ast D)^{k / 2} v}_{L^2}^2 + \norm{v}_{L^2}^2,\ \norm{w}_{\mscrW^\prime}^2 = \norm{(D^\ast D)^{(k - 1)/2} w}_{L^2}^2 + \norm{v}_{L^2}^2
\]
and denote the completions of $\mscrV^\prime$ and $\mscrW^\prime$ by $\mscrW^\prime_k$ and $\mscrV^\prime_{k - 1}$. By using the Kuiper theorem, we can trivialize $\mscrV^\prime_{k - 1}$ respecting the $Pin(2)$-action.

We explain the precise formulation. First divide $\mscrW^\prime_{k - 1}$ into three parts
\begin{align*}
  (\mscrW^\prime_{k - 1})_{\mcalH^+} & = \coprod_{(b, \bm{f}, a) \in \mbbB} \parenlr{H^1(\mbbX_b; i\R) \times \hplus{\mbbX_b; i\R}} / H^1(\mbbX_b; 2 \pi i\Z),                                     \\
  (\mscrW^\prime_{k - 1})_\R         & = \coprod_{(b, \bm{f}, a) \in \mbbB} \parenlr{H^1(\mbbX_b; i\R) \times \Omega^0_0(\mbbX_b; i\R) \times \Omega^+_0(\mbbX_b; i\R)} / H^1(\mbbX_b; 2 \pi i\Z), \\
  (\mscrW^\prime_{k - 1})_\C         & = \coprod_{(b, \bm{f}, a) \in \mbbB} \parenlr{H^1(\mbbX_b; i\R) \times \Gamma(S^-_b)} / H^1(\mbbX_b; 2 \pi i\Z).
\end{align*}
(To be precise, the right-hand side is actually the $L^2_{k-1}$-completion.) Here $\hplus{\mbbX_b; i\R}$ is a vector space consisting of purely imaginary self-dual harmonic 2-forms and $\Omega^+_0(\mbbX_b; i\R)$ is its $L^2$-orthogonal complement. The bundle $\Gamma(S^-_b)$ is acted on by $H^1(\mbbX_b; 2 \pi i \Z)$ at $(b, \bm{f}, a)$ using the canonical lift of $a$ to a harmonic $U(1)$-valued function described in \cref{prop:Alb}. Note that while $(\mscrW^\prime_{k - 1})_\C$ is acted on by $Pin(2)$, it is not a quaternionic Hilbert bundle since $j \in Pin(2)$ does not necessarily fix points of $\Pic^\mfraks(\mbbX_b)$. The Hilbert bundle $(\mscrW^\prime_{k - 1})_\R$ is a pullback of
\[
  (\mscrW^\pprime_{k - 1})_\R = \coprod_{b \in B} \Omega^0_0(\mbbX_b; i\R) \times \Omega^+_0(\mbbX_b; i\R),
\]
which is a Hilbert bundle over $B$.

We define orthogonal frame bundles
\begin{align*}
  \Fr((\mscrW^\pprime_{k - 1})_\R) & = \coprod_{b \in B} \{H_\R \to  (\mscrW^\pprime_{k - 1})_{\R, b}\ \text{metric preserving real liner isomorphisms}\},                   \\
  \Fr((\mscrW^\prime_{k - 1})_\C)  & = \coprod_{\bm{b} \in \mbbB} \{H_\HB \to  (\mscrW^\prime_{k - 1})_{\C, \bm{b}}\ \text{metric preserving complex linear isomorphisms}\}.
\end{align*}
Here for $H_\R$ and $H_\HB$ are a separable real Hilbert space and a separable quaternionic Hilbert space. These are a principal $O_\R(H_\R)$-bundle and $O_\C(H_\HB)$-bundle. (The bundle $(\mscrW^\prime_{k - 1})_\C$ is not quaternionic as a whole, and so we cannot construct a principal $O_\HB(H_\HB)$-bundle.)

The following lemma explains the precise meaning of the trivialization.

\begin{lemma}\label{lem:Kuiper}
  Assume that $B$ is contractible. Then there exist global sections
  \[
    \map{s_\R}{B}{\Fr((\mscrW^\pprime_{k - 1})_\R)},\ \map{s_\C}{\mbbB}{\Fr((\mscrW^\prime_{k - 1})_\C)}
  \]
  with the following property: for each $\bm{b} \in \mbbB$ and $h_\HB \in H_\HB$,
  \begin{align*}
    s_\C(\bm{b} \cdot j) (h_\HB \cdot j) = s_\C(\bm{b})(h_\HB) \cdot j
  \end{align*}
  holds.

  Furthermore, for each positive integer $n$, if
  \[
    \map{\bm{s}_\R}{\partial D^n \times B}{\Fr((\mscrW^\pprime_{k - 1})_\R)},\ \map{\bm{s}_\C}{\partial D^n \times \mbbB}{\Fr((\mscrW^\prime_{k - 1})_\C)}
  \]
  are maps which satisfy the above property for each point in $\partial D^n$, then these can be extended to maps from $D^n \times B$ or $D^n \times \mbbB$, holding the above property.
\end{lemma}

\begin{proof}
  The statement for $s_\R$ is obvious since $B$ is contractible.

  We prove the existence of $s_\C$. First note that $\Alb(\mbbX_b) \times Pic^\mfraks(\mbbX_b)$ is compact for each $b \in B$. Together with the assumption that $B$ is a contractible paracompact Hausdorff space, the fiber bundle $\mbbB \to B$ is trivial. So it suffices to show the statement assuming
  \[
    \mbbB = B \times \Alb(X) \times \Pic^\mfraks(X).
  \]
  Since $\mbbB$ is homotopy equivalent to the CW complex $\Alb(X) \times \Pic^\mfraks(X)$, the Kuiper theorem shows that $\Fr((\mscrW^\prime_{k - 1})_\C)$ admits a global section. However, it is not trivial that the section can be chosen to be $Pin(2)$-equivariant. So we have to carefully construct $s_\C$.

  Fix a homeomorphism
  \[
    \Pic^\mfraks(X) \approx T^4
  \]
  so that
  \[
    \{0, 1/2 \}^4
  \]
  corresponds to the fixed point set for the $j$-action. We endow $\Pic^\mfraks(X)$ with a CW structure so that any $k$-cell is a set of elements
  \[
    (x_1, x_2, x_3, x_4) \in T^4 = ([0, 1] / \{0, 1\})^4
  \]
  of the form $x_i = 0, 1 / 2$ or $x_i \in [0, 1 / 2], [1/2, 1]$ for each $i = 1, 2, 3, 4$.

  We construct $s_\C$ cell by cell. First, $(\mscrW^\prime_{k - 1})_\C$ is a quaternionic Hilbert bundle when restricted to
  \[
    B \times \Alb(X) \times \Pic^\mfraks(X)^{(0)}.
  \]
  (The notation $\Pic^\mfraks(X)^{(k)}$ stands for the $k$-skeleton.) Thus by the Kuiper theorem we obtain a section
  \[
    \map{s_\C^{(0)}}{B \times \Alb(X) \times \Pic^\mfraks(X)^{(0)}}{\Fr((\mscrW^\prime_{k - 1})_\C)}
  \]
  which is $Pin(2)$-equivariant. Next we extend $s_\C^{(0)}$ to a map
  \[
    \map{s_\C^{(1)}}{B \times \Alb(X) \times \Pic^\mfraks(X)^{(1)}}{\Fr((\mscrW^\prime_{k - 1})_\C)}.
  \]
  In order for $s_\C^{(1)}$ to be $Pin(2)$-equivariant, it suffices to construct an extension on
  \[
    B \times \Alb(X) \times (\Pic^\mfraks(X)^{(1)} \cap \{x_i \in [0, 1 / 2]\}),
  \]
  and the extension on
  \[
    B \times \Alb(X) \times (\Pic^\mfraks(X)^{(1)} \cap \{x_i \in [1 / 2, 1]\})
  \]
  is uniquely determined from that.

  Take a section $s_\C^{(1)\prime}$ of $\Fr((\mscrW^\prime_{k - 1})_\C)$ on
  \[
    B \times \Alb(X) \times \Pic^\mfraks(X)^{(1)}
  \]
  which is not necessarily equal to $s_\C^{(0)}$ on
  \[
    B \times \Alb(X) \times \Pic^\mfraks(X)^{(0)}.
  \]
  Then we have a continuous map
  \[
    \map{g}{B \times \Alb(X) \times \Pic^\mfraks(X)^{(0)}}{O_\C(H_\HB)}
  \]
  which satisfies
  \[
    s_\C^{(1)^\prime} g = s_\C^{(0)}.
  \]
  Using the contraction of $O_\C(H_\HB)$, we extend $g$ to a map on
  \[
    B \times \Alb(X) \times \Pic^\mfraks(X)^{(1)},
  \]
  and we put $s_\C^{(1)}= s_\C^{(1)^\prime} g$. The extensions to the higher dimensional skeletons are similar.

  Finally we show the extension property of maps from $\partial D^n \times \mbbB$ to $\Fr((\mscrW^\prime_{k - 1})_\C)$. Fix a $Pin(2)$-equivariant section $s_\C$ on $B$. By taking ratios, any $Pin(2)$-equivariant sections is identified with some $Pin(2)$-equivariant map from $\mbbB$ to $O_\C(H_\HB)$. The contraction of $O_\C(H_\HB)$ gives one of the desired extension.
\end{proof}

We can show the following lemmas, which is parallel to Furuta \cite[Lemma 3.2, 3.3]{Furuta-11/8-inequality2001}.

\begin{lemma}\label{lem:R and epsilon}
  Assume that $B$ is contractible and take sections $s_\R$ and $s_\C$ as in \cref{lem:Kuiper}. Fix a point $b_0 \in B$. Then there exist an open neighborhood $U^\prime$ of $b_0$, $R > 0$ and $\varepsilon > 0$ which satisfy the following:
  \begin{enumarabicp}
    \item For $v \in (\shplus{\mbbX} \times_B \mscrV^\prime_{k - 1})_{\bm{b}}$ with $\norm{v}_{\mscrV^\prime} \geq R$, $\bm{b} \in \restr{\mbbB}{U^\prime}$, we have
    \[
      \norm{F(v)}_{\mscrW^\prime} \neq 0.
    \]
    \item If further $R \leq \norm{v}_{\mscrV^\prime} \leq \sqrt{2} R$, then $\norm{F(v)}_{\mscrW^\prime} \geq \varepsilon$.
  \end{enumarabicp}
\end{lemma}

\begin{lemma}\label{lem:big fin-dim subsp}
  Assume that $B$ is contractible and take sections $s_\R$ and $s_\C$ as in \cref{lem:Kuiper}. Fix a point $b_0 \in B$ and take $U^\prime$, $R$ and $\varepsilon$ as in \cref{lem:R and epsilon}. Then there exist $Pin(2)$-equivariant finite-dimensional subspaces
  \[
    W_\R^\prime \subset H_\R,\ W_\HB^\prime \subset H_\HB
  \]
  and an open neighborhood $U \subset U^\prime$ of $b_0$ with the following properties:
  \begin{enumarabicp}
    \item The subspaces
    \[
      W_{\bm{b}} = \hplus{\mbbX_b} \oplus s_\R(\bm{b}) W_\R^\prime \oplus s_\C(\bm{b}) W_\HB
    \]
    and $\Image D$ span $\mscrW^\prime_{k - 1, \bm{b}}$ for each $\bm{b} = (b, \bm{f}, a) \in \restr{\mbbB}{U}$.
    \item An estimate
    \[
      \norm{(1 - p_{\bm{b}}) Q (v)}_{\mscrW^\prime} < \varepsilon
    \]
    holds for each $v \in (\shplus{\mbbX} \times_B \mscrV^\prime_{k - 1})_{\bm{b}}$, $R \leq \norm{v} \leq \sqrt{2} R$. $\bm{b} \in \restr{\mbbB}{U}$. Here $p_{\bm{b}}$ is the $L^2$-projection from to $\mscrW^\prime_{k - 1, \bm{b}}$ to $W_{\bm{b}}$.
  \end{enumarabicp}
\end{lemma}

The vector bundle $W = \coprod_{\bm{b} \in \mbbB} W_{\bm{b}}$ splits into two subbundles
\[
  \hplus{\mbbX} \oplus s_\R(W_\R^\prime),\ s_\C(W_\HB).
\]
In general, if a vector bundle $E \to B^\prime$ has a real summand and a complex summand, we write $E_\R$, $E_\R$ for each summand. Further if $E$ has a metric, we use the notation $B^\prime_R(E)$ for
\[
  B_R(E_\R) \times_{B^\prime} B_R(E_\C),
\]
where $B_R$ denotes a disk bundle of radius $R > 0$.

\begin{definition}\label{def:construction of FDA}
  Fix $b \in B$. Let $U^\prime$, $R$, $\varepsilon$, $U$ $W_\R^\prime$, $W_\HB$ as in \cref{lem:R and epsilon} and \cref{lem:big fin-dim subsp}. Let $V$, $W$ be finite-dimensional vector bundle over $\restr{\mbbB}{U}$ defined by
  \begin{align*}
    W & = \rho^\ast \hplus{\mbbX} \oplus \rho^\ast s_\R(W_\R^\prime) \oplus s_\C(W_\HB^\prime), \\
    V & = D^{-1}(W).
  \end{align*}
  Here $\rho$ is a projection from $\Pic^\mfraks(\restr{\mbbX}{U})$ onto $U$. We define a finite-dimensional approximation $F^\prime_f$ of $F^\prime$ by
  \[
    \map{F^\prime_f = p F^\prime}{\shplus{\restr{\mbbX}{U}} \times_U B_R^\prime(V)}{\shplus{\restr{\mbbX}{U}} \times_U W}.
  \]
\end{definition}

The map $F^\prime_f$ does not vanish on
\[
  \shplus{\restr{\mbbX}{U}} \times_U ((B_R(V_\R) \times_{\restr{\mbbB}{U}} \{0\}) \cup \partial B^\prime_R(V)).
\]

\subsection{Property of the finite-dimensional approximation map}

In this subsection, we abstract the property of the finite-dimensional approximation map constructed in \cref{subsec:SW map B_1 > 0}. The basic line is the same as \cite[Section 4.2]{Adachi-gerbe-like-2026}, except that the involved vector bundles are over $\Pic^\mfraks(\restr{\mbbX}{U})$.

\begin{definition}\label{def:model of FDA}
  A model of FDA for families of spin closed 4-manifolds consists of the following data:
  \begin{enumarabicp}
    \item  A topological space $\mcalU$.
    \item A finite-dimensional real vector bundle $H^+$ over $\mcalU$.
    \item A finite-rank free $\Z$-module bundle $H^1_\Z$ over $\mcalU$ and a vector bundle $H^1_\R = H^1_\Z \otimes \R$. We set $\Pic = H^1_\R / H^1_\Z$.
    \item Finite-dimensional real vector bundles $V_\R$, $W_\R$ with metrics over $\mcalU$.
    \item Finite-dimensional complex vector bundles $V_\C$, $W_\C$ with metric over $\Pic$. These bundles are equipped with metric-preserving $Pin(2)$-actions, where complex numbers act as scalar multiplications and the $j$-actions are anti-linear isomorphisms that cover the antipodal map on $\Pic$.
    \item An injective bundle map $\map{i}{H^+}{W_\R}$.
    \item $Pin(2)$-equivariant bundle maps $\map{D_\R}{V_\R}{W_\R}$, $\map{D_\C}{V_\C}{W_\C}$. The maps $D_\R$ and $i$ give a metric-preserving isomorphism
    \[
      V_\R \oplus H^+ \cong W_\R.
    \]
    \item A $Pin(2)$-equivariant fiberwise smooth map
    \[
      \map{F}{S(H^+) \times_\mcalU B^\prime(V)}{S(H^+) \times_\mcalU W}
    \]
    continuous with respect to the $C^\infty$ topology, where
    \[
      V =  V_\C \oplus \rho^\ast V_\R,\ W = W_\C \oplus \rho^\ast W_\R
    \]
    and $\rho$ is a projection from $\Pic$ onto $\mcalU$. The map $F$ does not vanish on
    \[
      S(H^+) \times_\mcalU ((B_R(V_\R) \times_{\restr{\mbbB}{\mcalU}} \{0\}) \cup \partial B^\prime_R(V)).
    \]
  \end{enumarabicp}
\end{definition}

\begin{example}\label{eg:FDA}
  We use the notations in \cref{subsec:SW map B_1 > 0}. The following tuple is an example of a model of FDA:
  \begin{align*}
    \mcalU & = \Alb(\restr{\mbbX}{U}),                                                \\
    H^+    & = \hplus{\restr{\mbbX}{U}},                                              \\
    H^1_\Z & = H^1(\mbbX; 2 \pi i \Z),                                                \\
    W_\R   & = \rho^\ast \hplus{\restr{\mbbX}{U}} \oplus \rho^\ast s_\R(W_\R^\prime), \\
    W_\C   & = s_\C(W_\HB^\prime),                                                    \\
    V      & =D^{-1}(W).                                                              \\
  \end{align*}
  Setting $\map{m_R}{V}{V}$ as the $R$-times map, we define
  \begin{align*}
    \text{``$D$''} & = D \circ m_R,          \\
    \text{``$F$''} & = F^\prime_f \circ m_R.
  \end{align*}
  (The left-hand side refers to the notation in \cref{def:model of FDA} and the right-hand side refers to the notation in \cref{def:construction of FDA}. We add the quotation mark in order to avoid confusion caused by the conflicting notation.) We replace the metric on $W_\R$ so that the isomorphism
  \[
    \map{D_\R \oplus i}{V_\R \oplus H^+}{W_\R}
  \]
  preserves the metrics.
\end{example}

When $X$ is a homology 4-torus with odd determinant, the finite-dimensional approximation map satisfies the additional properties.

\begin{definition}\label{def:model of FDA homology tori}
  A model of FDA for families of spin closed 4-manifolds
  \[
    \mscrF = (\mcalU, H^+, H^1_\Z, V_\R, W_\R, V_\C, W_\C, i, D_\R, D_\C, F)
  \]
  is called a model of FDA for families of homology 4-tori with odd determinant if it satisfies the following properties:
  \begin{enumarabicp}
    \item The rank of $H^+$ is 3 and $H^+$ is equipped with an orientation.
    \item The rank of $H^1_\Z$ is 4.
    \item For each $b \in \mcalU$, the cohomology class
    \[
      c(\mscrF_b) \in H^2(S(H^+_b); \Z)
    \]
    defined in \cref{rem:wcn} is an odd multiple of the generator.
  \end{enumarabicp}
\end{definition}

\begin{remark}\label{rem:wcn}
  The definition of $c(\mscrF_b)$ is as follows. To avoid notational complexity, assume that $\mcalU$ is a point and we drop $b$ from the notation. Define
  \[
    \map{F^\prime}{S(H^+) \times (S(V_\C) \times_{\Pic} B(V_\R)) \times [-1, 1]}{S(H^+) \times ((S(V_\C) \times_{\Pic} W_\C) \times W_\R)}
  \]
  by
  \[
    F^\prime(h, v_\C, v_\R, t) = (h, v_\C, w_\C, w_\R),
  \]
  where $w_\C$, $v_\R$ are elements that satisfy
  \[
    (h, w_\C, w_\R) = F\parenlr{h, \frac{1 + t}{2} v_\C, v_\R}.
  \]
  Taking the quotient by the $U(1)$-action, we obtain a map
  \[
    \map{\mcalF^\prime}{\mcalV^\prime}{\mcalW},
  \]
  where
  \begin{align*}
    \mcalV^\prime & = S(H^+) \times (S(V_\C) / U(1) \times_{\Pic} B(V_\R)) \times [-1, 1], \\
    \mcalW        & = S(H^+) \times ((S(V_\C) \times_{\Pic} W_\C) / U(1) \times W_\R)
  \end{align*}
  Viewing $\mcalW$ as a vector bundle over
  \[
    S(H^+) \times S(V_\C) / U(1),
  \]
  we have the Thom class $\tau_{\mcalW}$. Then we define
  \[
    c(\mscrF) = \int_{\mcalV^\prime \to S(H^+)} c_1(\mcalO(1)) \cdot \mcalF^{\prime\ast} \tau_{\mcalW}.
  \]
\end{remark}

\begin{proposition}
  If $X$ is a homology 4-torus with odd determinant, then the model of FDA constructed in \cref{eg:FDA} is the model of FDA for families of homology 4-tori with odd determinant.
\end{proposition}

\begin{proof}
  The conditions (1) and (2) are obvious. The condition (3) follows from the wall-crossing formula by Li--Liu \cite[Theorem 4.10]{Li--Liu-family-wall-crossing2001} or Baraglia--Konno \cite[Theorem 1.1, Theorem 1.8]{Baraglia--Konno-2022}.
\end{proof}

\section{Construction of spin structure}\label{sec:const of spin str}

The purpose of this section is to prove the following proposition, which is in line with \cite[Theorem 4.13]{Adachi-gerbe-like-2026}.

\begin{proposition}\label{prop:spin str from FDA}
  Let
  \[
    \mscrF = (\mcalU, H^+, H^1_\Z, V_\R, W_\R, V_\C, W_\C, i, D_\R, D_\C, F)
  \]
  be a model of FDA for families of homology 4-tori with odd determinant. Assume that $\mcalU$ is a locally simply-connected paracompact Hausdorff space. Then the following statements hold:
  \begin{enumarabicp}
    \item A spin structure $\mfrakt_\mscrF$ on $H^+$ is canonically constructed from $\mscrF$.
    \item Self-isomorphisms $\pm 1$ of $\mscrF$ induces self-isomorphisms of $\mfrakt_\mscrF$ covering the identity on $H^+$. By this correspondence, $+1$ on $\mscrF$ maps to $+1$ on $\mfrakt_{\mscrF}$, and $-1$ on $\mscrF$ maps to $-1$ on $\mfrakt_{\mscrF}$.
  \end{enumarabicp}
\end{proposition}

The basic strategy for the proof of \cref{prop:spin str from FDA} is the same as that of \cite[Theorem 4.13]{Adachi-gerbe-like-2026}: construct a triple for spin structure on $H^+$. Here, a triple for spin structure on a rank 3 vector bundle $H^+$ over $\mcalU$ is the triple
\[
  (H^+, L, \tiota)
\]
satisfying the following properties:
\begin{enumarabicp}
  \item $L$ is a complex line bundle over $S(H^+)$, and when restricted to the fiber $S(H^+_b)$ at $b \in \mcalU$, $c_1(L_b)$ is an odd multiple of the generator of $H^2(S(H^+_b); \Z)$.
  \item The map $\map{\tiota}{L}{L}$ is an anti-linear isomorphism that covers the antipodal map $\iota$ on $S(H^+)$. Further, $\tiota^2 = -1$.
\end{enumarabicp}

\begin{remark}
  In \cite[Section 3]{Adachi-gerbe-like-2026}, the definition of the triple for spin structure is slightly different. There, $c_1(L_b)$ should be the positive generator of $S(H^+_b)$. However, as mentioned in \cite[Remark 3.9]{Adachi-gerbe-like-2026}, the whole results in that section hold for $L$ with $c_1(L_b)$ an odd multiple of the generator.
\end{remark}

Following \cite[Subsection 4.3]{Adachi-gerbe-like-2026}, we construct a triple for spin structure as follows: take the $K$-theoretic degree of $\mcalF^\prime$ in \cref{rem:wcn}, and then take the determinant line bundle. The anti-linear $\Z/4$-action $\tiota$ comes from the action by $j \in Pin(2)$.

One difference from \cite[Subsection 4.3]{Adachi-gerbe-like-2026} is that the first Betti number of homology 4-tori is 4, not zero. Thus in order to take the $K$-theoretic degree of $\mcalF^\prime$, we have to make a choice of the spin structure on $H^1_\R$. It seems that there is no canonical choice of a spin structure on $H^1_\R$. However, we can prove the following result.

\begin{proposition}\label{prop:spin for H^1 is not relevant}
  Let
  \[
    \mscrF = (\mcalU, H^+, H^1_\Z, V_\R, W_\R, V_\C, W_\C, i, D_\R, D_\C, F)
  \]
  be a model of FDA for families of homology 4-tori with odd determinant. Assume that $\mcalU$ is a locally simply-connected paracompact Hausdorff space. Take a spin structure $\mfraku$ on $H^1_\R$ arbitrarily. Then the following statement holds:
  \begin{enumarabicp}
    \item A spin structure $\mfrakt_\mfraku$ on $H^+$ is canonically constructed.
    \item The construction of $\mfrakt_\mfraku$ is functorial with respect to $\mfraku$: if $\mfraku^\prime$ is another spin structure on $H^1_\R$ and $\map{\varphi}{\mfraku}{\mfraku^\prime}$ is an isomorphism, then $\varphi$ induces an isomorphism between $\mfrakt_\mfraku$ and $\mfrakt_{\mfraku^\prime}$. Furthermore, this construction is compatible with the composition of $\varphi$'s.
    \item Let $\pm 1$ be the self-isomorphisms of $\mfrakt_\mfraku$. Then the induced map of both of them is equal to the identity on $\mfrakt_\mfraku$.
  \end{enumarabicp}
\end{proposition}

The last statement of \cref{prop:spin for H^1 is not relevant} implies that the spin structures for different choices of $H^1_\R$ are glued together to give a canonical spin structure on $H^+$. Thus the resulting spin structure is independent of the choice of $\mfraku$.

In order to construct $\tiota$, we have to construct an anti-linear isomorphism on the spinor bundle for $H^1_\R$. This is accomplished in \cref{app:spin rep and anti-lin inv}.

In the following, we will prove \cref{prop:spin for H^1 is not relevant} and \cref{prop:spin str from FDA} under the assumption that $\mcalU$ is a point for notational simplicity. Modifications needed in the case of $\mcalU$ not necessarily a point is described in \cite[Section 4.4]{Adachi-gerbe-like-2026}.

Let
\[
  \mscrF = (\pt, H^+, H^1_\Z, V_\R, W_\R, V_\C, W_\C, i, D_\R, D_\C, F)
\]
be a model of FDA for families of homology 4-tori with odd determinant. As in \cref{rem:wcn}, we modify $F$ to a $\Z/2$-equivariant map
\[
  \map{F^\prime}{\mcalV^\prime}{\mcalW}.
\]
The $\Z/2$-action is defined by $j \in Pin(2)$. By using a $C^\infty$ function
\begin{align}\label{eq:rho}
  \map{f}{[0, \infty)}{[0, 1]}
\end{align}
which is identity near 0 and 1 for $s > 1$, we define
\[
  \map{\mcalF}{\mcalV}{\mcalW}
\]
by
\[
  \mcalF(h, [v_\C], v_\R, v_+) = \mcalF^\prime\parenlr{h, [v_\C], f(\abs{v_\R}) \frac{v_\R}{\abs{v_\R}}, f(\abs{v_+}) \frac{v_+}{\abs{v_+}}},
\]
where
\[
  \mcalV = S(H^+) \times (S(V_\C) / U(1) \times_{\Pic} V_\R) \times \R^+.
\]
(The symbol $\R^+$ indicates a 1-dimensional real vector space with a trivial $\Z/2$-action.)

We construct a family of Clifford bundles
\begin{align}\label{eq:Clifford bundle}
  (\mbfS, c, \nabla).
\end{align}
This splits into three parts
\[
  (S_{H^+}, c_{H^+}, \nabla_{H^+}),\ (S_\R, c_\R, \nabla_\R),\ (S_\C, c_\C, \nabla_\C).
\]
The resulting family of Clifford bundles is given by
\begin{align*}
  \mbfS  & = S_{H^+} \boxtimes S_\R \boxtimes S_\C,                                                                           \\
  c      & = c_{H^+} \otimes 1 \otimes 1 + \varepsilon \otimes c_\R \otimes 1 + \varepsilon \otimes \varepsilon \otimes c_\C, \\
  \nabla & = \nabla_{H^+} \otimes 1 \otimes 1 + 1 \otimes \nabla_\R \otimes 1 + 1 \otimes 1 \otimes \nabla_\C.
\end{align*}
Here $\varepsilon$ denotes the involution representing the $\Z/2$-grading. Before constructing the above, we must choose auxiliary data: a spin structure $\mfraku$ on $H^1_\R$ and a $\Z/2$-equivariant section
\begin{align}\label{eq:splitting}
  \map{s}{\pi_\C^\ast T \Pic}{T (S(V_\C) / U(1))}
\end{align}
of an exact sequence
\[
  0 \to T_{\Pic} (S(V_\C) / U(1)) \to T (S(V_\C) / U(1)) \to \pi_\C^\ast T \Pic \to 0,
\]
where
\[
  \map{\pi_\C}{S(V_\C) / U(1)}{\Pic}
\]
is the natural projection and
\[
  T_{\Pic} (S(V_\C) / U(1))
\]
is the tangent bundle along fibers. This data is necessary to construct the Clifford actions and the connections on $S_\C$ and $S_{\Pic}$.

First we define
\[
  S_{H^+} = \underline \C \oplus T S(H^+).
\]
(Note that $S(H^+)$ is an oriented 2-dimensional manifold, and so it admits a natural complex structure.) The family of Clifford multiplications and connections is uniquely determined, since
\[
  T_{S(H^+)} S(H^+),
\]
the tangent bundle along fibers, is a 0-dimensional vector bundle. Second, we define
\[
  S_\R = (V_\R \times \R^+) \times (\Lambda_\R^\ast V_\R \otimes_\R \C) \times (\Lambda_\R^\ast \R^+ \otimes_\R \C).
\]
The Clifford multiplication by $(v_\R, v_+) \in T (V_\R \times \R^+)$ is given by
\[
  (v_\R^\wedge - v_\R^\lrcorner) \otimes 1 + \varepsilon \otimes (v_+^\wedge - v_+^\lrcorner).
\]
Here $\varepsilon$ denotes the involution representing the $\Z/2$-grading. The connection $\nabla_\R$ is the Levi-Civita connection, or simply the trivial connection. Third, we define
\begin{align}\label{eq:S_C}
  S_\C & = (\underline \C \oplus \mcalO(1)) \otimes \Lambda_\C^\ast T_{\Pic} (S(V_\C)/U(1)) \otimes \pi_{\Pic}^\ast (S_{H^1_\R}^+ \oplus S_{H^1_\R}^-),
\end{align}
where $T_{\Pic} (S(V_\C)/U(1))$ is the tangent bundle along the fibers, $\map{\pi_{S(V_\C)}}{S(V_\C)}{\pt}$ is a projection and $S_{H^1_\R}^{\pm}$ is the spinor bundle defined in \cref{prop:spinor with anti-lin involution}. The Clifford multiplication by $v_\C \in T_{\Pic} (S(V_\C)/U(1))$ and $v_{\Pic} \in T \Pic$ is given by
\[
  \varepsilon \otimes (v_\C^\wedge - v_\C^\lrcorner),\ c_{H^1_\R}(v_{\Pic}).
\]
Together with the fixed splitting $s$, these give a Clifford multiplication by an element in $T (S(V_\C)/U(1))$ on $S_\C \boxtimes S_{\Pic}$. The connection $\nabla_\C$ is constructed from
\begin{itemize}
  \item the canonical connection on the fibers of $S(V_\C) \to \Pic$,
  \item and the trivial connection on $\Pic$.
\end{itemize}
(Note that the Fubini--Study metric is invariant under the action of $U(V_\C)$.) This finishes the definition of $(\mbfS, c, \nabla)$.

In order to construct a family of Fredholm operators over $S(H^+)$, we have to construct a family of Hermitian maps
\begin{align}\label{eq:h}
  \mbfh \in \Gamma(\mcalV; \mbfS)
\end{align}
with compact support. Again we construct
\[
  h_{H^+} \in \Gamma(\mcalV; S_{H^+}),\ h_\R \in \Gamma(\mcalV; S_\R),\ h_\C \in \Gamma(\mcalV; S_\C)
\]
respectively, and $\mbfh$ is defined as
\[
  \mbfh = h_{H^+} \otimes 1 \otimes 1 + \varepsilon \otimes h_\R \otimes 1 + \varepsilon \otimes \varepsilon \otimes h_\C.
\]
Here, by $\Gamma(\mcalV; S_{H^+})$ (resp. $\Gamma(\mcalV; S_\R)$ and $\Gamma(\mcalV; S_\C))$ we mean the section of $S_{H^+}$ (resp. $S_\R$ and $S_\C$) pulled back to $\mcalV$.

To define $h_{H^+}$, $h_\R$ and $h_\C$, we first fix an element $v$ of $\mcalV$. By sending $v$ by $\mcalF$, we obtain $\mcalF(v) \in \mcalW$. Noting that there is a canonical isomorphism
\[
  S(H^+) \times W_\R \cong T S(H^+) \oplus (S(H^+) \times \R^+) \oplus (S(H^+) \times V_\R),
\]
the element $w \in \mcalW$ splits into the three parts: $w_{H^+} \in T S(H^+)$, $w_\R \in \R^+ \times V_\R$ and $w_\C \in (S(V_\C) \times_{\Pic} W_\C)$.

Bearing this in mind, we define

\begin{align*}
  h_{H^+}(v) & = i(\mcalF(v)_{H^+}^\wedge - \mcalF(v)_{H^+}^\lrcorner),                         \\
  h_\R(v)    & = \mcalF(v)_\R^\wedge + \mcalF(v)_\R^\lrcorner,                                  \\
  h_\C(v)    & = \varepsilon \otimes i(\mcalF(v)_\C^\wedge - \mcalF(v)_\C^\lrcorner) \otimes 1.
\end{align*}

Finally, the $\Z/4$-actions
\begin{align*}
   & \map{\tau_{H^+}}{S_{H^+}}{S_{H^+}}, \\
   & \map{\tau_\R}{S_\R}{S_\R},          \\
   & \map{\tau_\C}{S_\C}{S_\C}
\end{align*} are defined by the action of $j \in Pin(2)$. (Note that $j$ acts on $S(V_\C) \times \C$ by $(v_\C, z) \cdot j = (v_\C \cdot j, \bar{z})$.) The $\Z/4$-action $\tau$ on $\mbfS$ is defined by
\begin{align}\label{eq:tau}
  \tau = \tau_{H^+} \otimes \tau_\R \otimes \tau_\C.
\end{align}

Let $\mcalD$ denote the family of the Dirac operators constructed from $(\mbfS, c, \nabla)$. From construction $\mcalD$ and $\tau$ is commutative. We have the following proposition.

\begin{proposition}\label{prop:const of spin str}
  For sufficiently large $t > 0$, the operator
  \[
    \mcalD + t\mbfh
  \]
  is a family of Fredholm operators. The map $\tau$ induces a $\Z/4$-action $\tiota$ on the determinant line bundle
  \[
    \det(\mcalD + t\mbfh).
  \]
  The triple $(H^+, \det(\mcalD + t\mbfh), \tiota)$ is a triple for spin structure on $H^+$. (For the definition of triples for spin structure, see \cite[Definition 3.2]{Adachi-gerbe-like-2026}.)
\end{proposition}

\begin{proof}
  The families index theorem and the assumption in \cref{def:model of FDA homology tori} shows that
  \[
    c_1(\det(\mcalD + t\mbfh))
  \]
  is an odd multiple of the generator. Since $\tiota^4 = 1$, we have $\tiota^2 = \pm 1$. However, the fact that $c_1(\det(\mcalD + t\mbfh))$ is odd forces $\tiota^2 = -1$.
\end{proof}

\begin{proof}[Proof of \cref{prop:spin for H^1 is not relevant}]
  The statement in (1) is already proved, and the statement in (2) follows easily from the construction. For (3), by \cref{prop:spinor with anti-lin involution} we first note that $-1$, which is the isomorphism of $\mfrakt_\mfraku$, acts on $S_{H^1_\R}^\pm$ as the multiplication by $-1$ and thus on $\mbfS$ by $-1$. On the other hand, the virtual rank of $\mcalD + t\mbfh$ is 0 by the index theorem. This shows that the multiplication by $-1$ on $\mbfS$ gives a self-isomorphism
  \[
    \map{(-1)^0 = 1}{\det(\mcalD + t\mbfh)}{\det(\mcalD + t\mbfh)}.
  \]
  This finishes the proof of (3).
\end{proof}

\begin{definition}\label{def:spin str}
  We denote the spin structure on $H^+$ constructed in \cref{prop:const of spin str} by $\mfrakt_{\mscrF}$.
\end{definition}

\begin{remark}
  The above argument actually shows that the Seiberg--Witten invariants for spin structures on homology 4-tori with odd determinant are odd, which is first proved by Ruberman--Strle \cite{Ruberman--Strle-mod2-SW-homology-tori-2000}. More generally, as we will see in \cref{sec:mod 2 SW}, a similar construction computes the mod 2 Seiberg--Witten invariants for spin structures on closed 4-manifolds with $b^+ \geq 3$, which recovers part of the result by Baraglia \cite{baraglia-mod-2023-arxiv}.
\end{remark}

The following proposition states that the spin structure $\mfrakt_{\mscrF}$ does not depend on the various choices of auxiliary data other than the model of FDA $\mscrF$.

\begin{proposition}\label{prop:indep of auxiliary data}
  The spin structure $\mfrakt_{\mscrF}$ on $H^+$ does not depend on the choices of the following data:
  \begin{itemize}
    \item the smooth function $f$ in \cref{eq:rho},
    \item the splitting $s$ in \cref{eq:splitting}, and
    \item the sufficiently large number $t > 0$.
  \end{itemize}
  More precisely, if $(f_0, s_0, t_0)$, $(f_1, s_1, t_1)$ are triples and $\mfrakt_{\mscrF, 0}$, $\mfrakt_{\mscrF, 1}$ are spin structures constructed from these, then there is a canonical isomorphism
  \[
    \map{\Phi_{10}}{\mfrakt_{\mscrF, 0}}{\mfrakt_{\mscrF, 1}},
  \]
  and if $(f_2, s_2, t_2)$ is a third triple and $\mfrakt_{\mscrF, 2}$ is the spin structure constructed from that, then we have
  \[
    \Phi_{20} = \Phi_{21} \circ \Phi_{10}.
  \]
\end{proposition}

\begin{proof}
  This follows from \cite[Proposition 3.12]{Adachi-gerbe-like-2026}, which states that a family of triples for spin structure over some simply-connected base space gives a canonically isomorphic family of spin structures.
\end{proof}

\begin{proof}[Proof of \cref{prop:spin str from FDA}]
  The statement (1) follows from \cref{prop:const of spin str}. The spin structure is independent of the choice of spin structures on $H^1_\R$ by \cref{prop:spin for H^1 is not relevant}.

  Next we prove the statement (2). A self-isomorphism of the model of FDA $\mscrF$ gives a self-isomorphism of the family of the functional spaces on which $\mcalD + t\mbfh$ acts. It commutes with $\mcalD + t\mbfh$, and so it gives a self-isomorphism of $\det(\mcalD + t\mbfh)$. This gives a self-isomorphism of $\mfrakt_{\mscrF}$.

  For the latter statement, it suffices to show that $-1$ corresponds to $-1$. The self-isomorphism $-1$ of $\mscrF$ is equal to the action by $j^2 \in Pin(2)$, and so the induced self-isomorphism is also equal to $\tiota^2$, which is $-1$.

\end{proof}

As for the choice of the finite-dimensional approximation, we have the following proposition. We omit the proof since it is essentially the same as that of \cite[Theorem 5.1]{Adachi-gerbe-like-2026}.

\begin{proposition}\label{prop:indep of FDA}
  Let $\mscrF = (\pt, H^+, H^1_\Z, V_\R, W_\R, V_\C, W_\C, i, D_\R, D_\C, F)$ be a model of FDA for families of homology 4-tori with odd determinant over a point. Take a real vector space $V_\R^\prime$ with a metric and a complex vector bundle $W_\C^\prime$ with a metric and a $Pin(2)$-action over $\Pic$. Denote by $\mscrF^\prime$ the model of FDA obtained by taking a direct sum of $\mscrF$ and $V_\R^\prime$, $W_\C^\prime$. Then there is a canonical isomorphism
  \[
    \map{\Psi_{\mscrF^\prime \mscrF}}{\mfrakt_{\mscrF^\prime}}{\mfrakt_{\mscrF}}.
  \]
  Furthermore, this isomorphism is compatible: if $\mscrF^\pprime$ is a third model of FDA obtained by taking a direct sum of $\mscrF^\prime$, a vector space $V_\R^\pprime$ with a metric and a complex vector bundle $W_\C^\pprime$ with a metric and a $Pin(2)$-action over $\Pic$, then
  \[
    \Psi_{\mscrF^\pprime \mscrF} = \Psi_{\mscrF^\pprime \mscrF^\prime} \circ \Psi_{\mscrF^\prime \mscrF}.
  \]
\end{proposition}

\section{Proof of the main theorems}\label{sec:proof of the main theorems}

We give the proof of \cref{thm:Spin(3) from Diffspin}.

\begin{proof}[Proof of \cref{thm:Spin(3) from Diffspin}]
  Since the principal $\Diffspiniso$-bundle $\mcalE \to B$ has a lift $\tmcalE$ to a principal $\Diffspin$-bundle and we have assumed that $B$ is contractible, we can apply \cref{lem:Kuiper}. This lead to the fact that $B$ can be covered by open sets $\{U_\alpha\}_\alpha$ such that for each $\alpha$ there is a model of FDA on $\Alb(\restr{\mbbX}{U_\alpha})$ as in \cref{def:construction of FDA}.

  By \cref{prop:const of spin str} and \cref{prop:indep of auxiliary data}, we have a spin structure $\mfrakt_\alpha$ on
  \[
    \restr{\pi_{\Alb(\mbbX)}^\ast H^+}{\Alb(\restr{\mbbX}{U_\alpha})}.
  \]
  In \cref{sec:const of spin str}, we only discuss the case of the base space a point, but the argument also works in the general case by slight modifications. This spin structure is also independent of the choice of the model of FDA by \cref{prop:indep of FDA}.

  Finally, \cref{prop:spin str from FDA}(2) implies that $\mfrakt_\alpha$'s glue together to give a canonical spin structure on $H^+$ over the whole base space $B$. This finishes the proof.
\end{proof}

Next we prove \cref{thm:bd Dehn twist}. First, we describe the criterion whether the boundary Dehn twist is not isotopic to the identity. This is the slight modification of a statement by Y. Lin \cite[Proposition 2.1]{Lin-bd-Dehn2025}.

\begin{proposition}\label{prop:Dehn twist vs commutator}
  Let $X$ be a closed connected 4-manifold with possibly $b_1 > 0$ and set $X^\prime = X \setminus \Int D^4$. Then the following statements are equivalent:
  \begin{enumarabicp}
    \item There exist a smooth $X$-bundle $X \to \mbbX \to \Sigma_g$ over a closed oriented surface of genus $g \geq 0$ and a global section $\map{s}{\Sigma_g}{\mbbX}$ with $s^\ast w_2(T_{\Sigma_g} \mbbX) \neq 0$.
    \item There exist $2g$ boundary-fixing diffeomorphisms $a_1, b_1, \cdots, a_g, b_g$ of $X^\prime$ such that the boundary Dehn twist of $X^\prime$ is isotopic to
    \[
      [a_1, b_1] \cdots [a_g, b_g]
    \]
    relative to $\partial X^\prime$, where $[\cdot, \cdot]$ denotes the commutator.
  \end{enumarabicp}
\end{proposition}

The proof is the same as that of \cite[Proposition 2.1]{Lin-bd-Dehn2025}.

\begin{proof}[Proof of \cref{thm:bd Dehn twist}]
  By \cref{prop:Dehn twist vs commutator} it suffices to show that for every smooth $X$-bundle  $X \to \mbbX \to S^2$ and every section $\map{s}{S^2}{\mbbX}$, we have $s^\ast w_2(T_{S^2} \mbbX) = 0$. Choose a spin structure on $X$. From the existence of the global section $s$, we can construct a continuous map $S^2 \to \Alb(\Xunivspiniso)$ whose composition with $\Alb(\Xunivspiniso) \to B\Diffspiniso$ is a classifying map for $\mbbX$. Thus we have
  \[
    \alpha(\mbbX, \mfraks) = w_2(\hplus{\mbbX}).
  \]
  The clutching function for $H^2(\mbbX; \R) \to S^2$ can be taken to be the identity, and so the same is true for $\hplus{\mbbX}$. Therefore $w_2(\hplus{\mbbX})$ vanishes, and by the above equality $T_{S^2} \mbbX$ admits a spin structure. This implies $s^\ast w_2(T_{S^2} \mbbX) = 0$ and the proof of \cref{thm:bd Dehn twist} is finished.
\end{proof}

\section{The mod 2 Seiberg--Witten invariants for closed spin 4-manifolds with $b^+ \geq 3$}\label{sec:mod 2 SW}

The purpose of this section is to give an alternative proof of the calculation of the mod 2 Seiberg--Witten invariants for closed spin 4-manifolds with $b^+(X) \geq 3$, which is originally due to Baraglia \cite{baraglia-mod-2023-arxiv}.

We briefly review the definition of the Seiberg--Witten invariants taking values in the cohomology of $\Pic^\mfraks(X)$. For a detailed explanation, see Baraglia \cite[Section 1.1]{baraglia-mod-2023-arxiv} for example. We always take the coefficient as $\Z/2$. Let $X$ be a closed oriented 4-manifold possibly with $b_1(X) > 0$ and $\mfraks$ be a spin structure on $X$. From the Seiberg--Witten equations (which is perturbed by an element in $\Omega^+(X)$), we obtain the following data:
\begin{itemize}
  \item The moduli space $\mcalM$ and the principal $U(1)$-bundle $\tilde\mcalM \to \mcalM$ and
  \item the ``projection'' map $\map{p}{\mcalM}{\Pic^\mfraks(X)}$.
\end{itemize}
We set $x = c_1(\tilde\mcalM)$. The (mod 2) Seiberg--Witten invariants are defined by
\[
  \SW_{X, \mfraks}(x^k) = p_! x^k \in H^{2k - d}(\Pic^\mfraks(X); \Z/2),
\]
where $d$ is the dimension of the moduli space $\mcalM$.

Our aim is to prove the following result due to Baraglia \cite{baraglia-mod-2023-arxiv}.

\begin{theorem}\label{thm:mod 2 SW}
  Let $X$ be a closed oriented 4-manifold with $b^+(X) \geq 3$ and $\mfraks$ be a spin structure on it. The mod 2 Seiberg--Witten invariant for $X, \mfraks$ is computed as follows.
  \begin{enumarabicp}
    \item If $k > 0$, then
    \[
      \SW_{X, \mfraks}(x^k) = 0.
    \]
    \item If $b^+(X) = 3$, then
    \[
      \SW_{X, \mfraks}(1) = s_{2 + \sigma(X) / 8}(D_\mfraks),
    \]
    where $D_\mfraks$ is the families Dirac operator parametrized by $\Pic^\mfraks(X)$ and $s_\ast(D_\mfraks)$ is the Segre class.
    \item If $b^+(X) > 3$, then $\SW_{X, \mfraks}(1) = 0$.
  \end{enumarabicp}
\end{theorem}

Baraglia \cite{baraglia-mod-2023-arxiv} computes the mod 2 Seiberg--Witten invariant even in the case of $b^+(X) < 3$, and so the argument below does not fully recover his result. His strategy is to construct an invariant coming from the $Pin(2)$-equivariance of the Seiberg--Witten map, which includes the information of the usual mod 2 Seiberg--Witten invariant.

Our method here is based on the following proposition.

\begin{proposition}\label{prop:top restr for anti-lin map}
  Let $m \geq 2$ be an integer, and $L$ be a complex line bundle over $S^m$. Suppose that
  \[
    \map{\tiota}{L}{L}
  \]
  is an anti-linear isomorphism which
  \begin{itemize}
    \item covers the antipodal map on $S^m$, and
    \item squares to $\pm 1$.
  \end{itemize}
  Then we have
  \[
    \tiota^2 = \begin{cases*}
      1             & if $m \geq 3$, \\
      (-1)^{c_1(L)} & if $m = 2$.
    \end{cases*}
  \]
\end{proposition}

\begin{proof}
  Since the $m \geq 3$ case can be derived from the $m = 2$ case, we assume that $m = 2$. It is easy to see that for each $c_1(L)$, we can construct an anti-linear isomorphism $\map{\tiota_0}{L}{L}$ which covers the antipodal map and squares to $(-1)^{c_1(L)}$. We can write
  \begin{align*}
    \tiota = f \tiota_0, &  & \map{f}{S^2}{\C^\times}.
  \end{align*}
  The condition $\tiota^2 = \pm 1$ is equivalent to
  \begin{align}\label{eq:tiota^2}
    f(-x) \overline{f(x)} = \pm 1, &  & x \in S^2.
  \end{align}
  Since $S^2$ is simply-connected, there are continuous functions
  \[
    \map{r}{S^2}{(0, \infty)},\ \map{\theta}{S^2}{\R}
  \]
  such that $f = r e^{i\theta}$. Then the condition \cref{eq:tiota^2} is rewritten as
  \[
    r(x) r(-x) = 1,\ \theta(-x) - \theta(x) = \pi k\ \text{for some $k \in \Z$}.
  \]
  By substituting $-x$ for $x$ in the above expression, we have
  \[
    \theta(x) - \theta(-x) = \pi k.
  \]
  Summing up the two equality, we have
  \[
    0 = (\theta(-x) - \theta(x)) + (\theta(x) - \theta(-x)) = 2k,
  \]
  leading to $k = 0$. This means $\tiota^2 = \tiota_0^2 = (-1)^{c_1(L)}$.
\end{proof}

Our goal is to construct a complex line bundle over $\shplus{X}$ and an anti-linear map $\tiota^2$ on it with $\tiota^2$ equal to $(-1)^{\SW_{X, \mfraks}(x^m)}$. The basic idea of the construction is the same as the construction in \cref{prop:const of spin str}: take the determinant line bundle of the $K$-theoretic degree of the finite-dimensional approximation of the Seiberg--Witten map deformed and quotiented by $U(1)$. We note that there is three main differences.
\begin{itemize}
  \item We only need to consider the case in which the base space $B$ is a point.
  \item The construction need not be canonical since our purpose is just to compute the numerical invariant.
  \item We will use $(\C \oplus \mcalO(1))^{\otimes m + 1}$ instead of $\C \oplus \mcalO(1)$ in \cref{eq:S_C}. This is because we want $SW_{X, \mfraks}(x^m)$ to appear as the virtual rank of the $K$-theoretic degree, which is calculated through the index theorem.
  \item We do not necessarily have a canonical choice of the spin$^c$ structure on $T\shplus{X}$. Instead, we will take a subtorus $T$ of $\Pic^\mfraks(X)$ and choose the spin$^c$ structure on
        \[
          T(\shplus{X} \times T)
        \]
        together with an anti-linear $\Z/2$-action. This is useful since $\shplus{X}$ and $T$ is not necessarily even-dimensional in general, but their product is always even-dimensional given that the Seiberg--Witten invariant evaluated by $T$ is nonzero.
\end{itemize}

As mentioned above, the following lemma is key.

\begin{lemma}\label{lem:spinor of S^m times T^n}
  Let $m \geq 2$, $n$ be nonnegative integers with $m + n$ even. Let
  \[
    \map{\iota}{S^m \times T^n}{S^m \times T^n}
  \]
  be an involution defined by
  \begin{align*}
    \iota(x, y) & = (-x, -y), & x \in S^m,\ y \in T^n = (\R/\Z)^n.
  \end{align*}
  Fix a spin structure on
  \[
    T(S^m \times T^n) \to S^m \times T^n
  \]
  in the following way: take the unique (up to isomorphism) spin structure on $S^m$, and take a trivial spin structure on $T^n$ associated with the trivialization of $T T^n$ coming from the Lie group structure. Let $C$ be an anti-linear self-isomorphism of $S^+ \oplus S^-$ associated with the spin structure. (The map $C$ commutes with the Clifford multiplications and $C^2 = (-1)^{\frac{1}{8}(m + n)(m + n + 2)}$.) Then, there is an anti-linear isomorphism
  \[
    \map{\tau}{S^+ \oplus S^-}{S^+ \oplus S^-}
  \]
  (where $S^\pm$ denote the spinor bundles) covering $\iota$ and commuting with $C$, which makes the following diagram commutative.
  \[
    \begin{tikzcd}[column sep=large]
      T(S^m \times T^n) \times_{S^m \times T^n} (S^+ \oplus S^-) \arrow[r, "c"] \arrow[d, "\iota \times \tau"] & S^+ \oplus S^- \arrow[d, "\tau"] \\
      T(S^m \times T^n) \times_{S^m \times T^n} (S^+ \oplus S^-) \arrow[r, "c"] & S^+ \oplus S^-
    \end{tikzcd}
  \]
  Here $c$ denotes the Clifford multiplication. The map $\tau$ is unique up to $\{\pm 1\}$. Furthermore, $\tau^2 = (-1)^{\frac{1}{8}(m + n)(m + n + 2)}$.
\end{lemma}

\begin{proof}
  The uniqueness follows from Schur's lemma.

  The existence is proved by an explicit construction. Let
  \[
    \map{c_0}{Cl(\R^{m + 1 + n})}{\End(S_0)}
  \]
  be an irreducible representation of the Clifford algebra. Then there is an anti-linear operator
  \[
    \map{C_0}{S_0}{S_0}
  \]
  which is
  \begin{itemize}
    \item commutative with the Clifford multiplications by elements in $\R^{m + 1 + n}$ if $(m + n) / 2$ is odd, and
    \item anti-commutative with the Clifford multiplications by elements in $\R^{m + 1 + n}$ otherwise.
  \end{itemize}
  The operator $C_0$ can be taken so that $C_0^2 = -1$ if $m + n \equiv 2, 4 \mod 8$ and $C_0^2 = 1$ if $m + m \equiv 6, 8 \mod 8$.

  We define $S^\pm$ as follows. For each $(x, y) \in S^{m + 1} \times T^n$, take an oriented orthonormal basis $x, e_2, \cdots, e_{m + 1 + n}$ of $\R^{m + 1 + n}$. (We regard $x$ as a unit vector in $\R^{m + 1 + n}$.) Then $S^\pm_{(x, y)}$ is the $\pm 1$-eigenvector space of the Clifford multiplication on $S_0$ by
  \[
    i^{(m + n)/2} (x e_2) \cdots (x e_{m + 1 + n}) \in Cl(\R^{m + 1 + n}).
  \]
  Especially, we have a canonical isomorphism
  \[
    S^+ \oplus S^- \cong (S^m \times T^n) \times S_0.
  \]
  The Clifford multiplication by $v \in T_{(x, y)} (S^m \times T^n) \subset \R^{m + 1 + n}$ is given by
  \[
    c_{(x, y)}(v) = c_0(x) c_0(v).
  \]
  The operator $C$ is defined by
  \begin{align*}
    C(x, y, \phi) = (x, y, C_0 \phi), &  & x \in S^m,\ y \in T^n,\ \phi \in S_0.
  \end{align*}

  With this setup, we define $\tau$ by
  \[
    \tau(x, y, \phi) = (-x, -y, C_0 \phi).
  \]
  It commutes with $C$. If $v \in T_{(x, y)} (S^m \times T^n)$, then
  \begin{align*}
    \tau c_{(x, y)}(v) (\phi) & = (-x, -y, C_0 c_0(x)c_0(v) \phi)    \\
                              & = (-x, -y, c_0(-x) c_0(-v) C_0 \phi) \\
                              & = c_{(-x, -y)(-v)} \tau (\phi),
  \end{align*}
  which is the desired commutativity. Finally, we have $\tau^2 = C_0^2$.
\end{proof}

\begin{proof}[Proof of \cref{thm:mod 2 SW}]
  If $k$ is odd, then the charge conjugation symmetry of the (integer-valued) Seiberg--Witten invariant implies that $\SW_{X, \mfraks}(x^k) = 0$; see Baraglia \cite[Lemma 3.3]{baraglia-mod-2023-arxiv} for a detail. In the following, we assume that $k$ is even.

  It suffices to compute $\SW_{X, \mfraks}(x^k)[T]$ for every subtorus $T$ of $\Pic^\mfraks(X)$. By the dimensional reason, it can be nonzero only when
  \[
    2k + \frac{\sigma(X)}{8} + b^+(X) + 1 - \dim T = 0,
  \]
  and especially $\dim \shplus{X} + \dim T$ is even. We proceed with this assumption. By mimicking the construction of \cref{eq:Clifford bundle} with three modifications, we have a Clifford bundle
  \[
    (\mbfS, c, \nabla).
  \]
  The modifications are as follows. First, we replace $\C \oplus \mcalO(1)$ in $S_\C$ by $(\C \oplus \mcalO(1))^{\otimes k + 1}$. Second, we replace $\Pic$ by $T$. Third, we replace the tensor product of $S_{H^+}$ and $S_{H^1_\R}^+ \oplus S_{H^1_\R}^-$ in $S_\C$ by the spinor bundle described in \cref{lem:spinor of S^m times T^n}. We also have the Hermitian map $\mbfh$ in \cref{eq:h} and the anti-linear symmetry $\tau$ by mimicking \cref{eq:tau}.

  Decompose $(\C \oplus \mcalO(1))^{\otimes k + 1}$ into two parts:
  \[
    \mcalO_+ = \bigoplus_{\text{$j$ even}} \binom{k + 1}{j} \mcalO(j),\ \mcalO_- = \bigoplus_{\text{$j$ odd}} \binom{k + 1}{j} \mcalO(j).
  \]
  According to this decomposition, $(\mbfS, c, \nabla)$, $\mbfh$ and $\tau$ is also split into
  \begin{align*}
    (\mbfS_+, c_+, \nabla_+),\ \mbfh_+,\ \tau_+,\ (\mbfS_-, c_-, \nabla_-),\ \mbfh_-, \tau_-.
  \end{align*}
  Thus we have an isomorphism between the determinant line bundle
  \[
    \det(\mcalD + t \mbfh) \cong \det(\mcalD_+ + t \mbfh_+) \otimes \det(\mcalD_- + t \mbfh_-)
  \]
  for some sufficiently large $t > 0$. Furthermore, $\tau$ and $\tau_{\pm}$ induce anti-linear isomorphisms $\tiota$ and $\tiota_{\pm}$ on the above determinant line bundles, and they satisfy $\tiota = \tiota_+ \otimes \tiota_-$.

  To apply \cref{prop:top restr for anti-lin map}, we have to compute and $\tiota^2$ and $c_1(L)$ if $b^+(X) = 3$. First we compute $\tiota^2$. By the above remark, it suffices to compute $\tiota_+^2$ and $\tiota_-$. A simple calculation shows that exactly one of $\tau_+^2$ and $\tau_-^2$ is $+1$ and the other is $-1$. Furthermore, we will see both of the parities of the virtual ranks of
  \[
    \ind(\mcalD_+ + t \mbfh_+),\ \ind(\mcalD_- + t \mbfh_-)
  \]
  are the same as $\SW_{X, \mfraks}(x^k)[T]$. This then leads to
  \[
    \tiota^2 = \tiota_+^2 \otimes \tiota_-^2 = (-1)^{\SW(X, \mfraks)[T]}.
  \]

  To compute the parities of the above virtual ranks, we show that the virtual rank of $\ind(\mcalD + t \mbfh)$ is 0 and the parity of the virtual rank of $\ind(\mcalD_- + t \mbfh_-)$ is $\SW_{X, \mfraks}(x^k)[T]$. The first half of the statement is derived directly from the index theorem. For the second half of the statement, we first note that the parity of the virtual rank of $\ind(\mcalD_- + t \mbfh_-)$ can be expressed in a different manner. That is, instead of using $\mcalO_-$, we use
  \[
    \mcalO_0 = \mcalO(1) \otimes (\C \oplus \mcalO(1))^{\otimes k}
  \]
  and construct
  \[
    (\mbfS_0, c_0, \nabla_0), \mbfh_0.
  \]
  Then the parity of the virtual rank of $\ind(\mcalD_- + t \mbfh_-)$ is equal to that of $\ind(\mcalD_0 + t \mbfh_0)$. This is because
  \[
    \text{(the coefficient of $\mcalO(j)$ in $\mcalO_-$)} - \text{(the coefficient of $\mcalO(j)$ in $\mcalO_0$)}
  \]
  is even for every $j = 0, 1, \cdots, k + 1$. (Here we use the assumption that $k$ is even.) Then the index theorem tells us that the virtual rank of $\ind(\mcalD_0 + t \mbfh_0)$ is exactly $\SW_{X, \mfraks}(x^k)$.

  Finally, we compute $c_1(L)$ when $b^+(X) = 3$. By the families index theorem, $c_1(L)$ is equal to the $k + 1$-th families Seiberg--Witten invariant with respect to the tautological chamber of $\shplus{X} \times \hplus{X}$ defined and computed in Baraglia--Konno \cite[Definition 2.7, Theorem 3.6]{Baraglia--Konno-2022}. It is equal to
  \[
    s_{2 + \sigma(X) / 8}(D_{\mfraks})[T],
  \]
  the Segre class of the families Dirac operator evaluated by $[T]$, by the family version of the wall-crossing formula (which is originally due to Li--Liu \cite[Corollary 4.10]{Li--Liu-family-wall-crossing2001}, and is given an alternative proof by Baraglia--Konno \cite[Theorem 5.2]{Baraglia--Konno-2022}). Now \cref{prop:top restr for anti-lin map} gives the desired result.
\end{proof}

\appendix
\section{Spinor representations in even dimensions and the anti-linear involution compatible with Clifford multiplication}\label{app:spin rep and anti-lin inv}

The purpose of this section is to prove the following proposition.

\begin{proposition}\label{prop:spinor with anti-lin involution}
  Let $E \to B$ be a $2n$-dimensional real oriented vector bundle over a topological space $B$. Suppose that a spin structure on $E$ are given and denote the spinor bundles by $S_E^\pm$. Then there exists an anti-linear isomorphism
  \[
    \map{\tau_E}{S_E^+ \oplus S_E^-}{S_E^+ \oplus S_E^-}
  \]
  that preserves the Real part of the metric on $S_E^\pm$ and makes the following diagram commutative.
  \[
    \begin{tikzcd}[column sep=large]
      E \times_B (S_E^+ \oplus S_E^-) \arrow[r, "c_E"] \arrow[d, "(-\id)\times \tau_E"] & S_E^+ \oplus S_E^- \arrow[d, "\tau_E"] \\
      E \times_B (S_E^+ \oplus S_E^-) \arrow[r, "c_E"] & S_E^+ \oplus S_E^-
    \end{tikzcd}
  \]
  Here $c_E$ denotes the Clifford multiplication. If $n$ is even, $\tau$ preserves the grading of $S_E^\pm$ and if $n$ is odd $\tau$ reverses the grading of $S_E^\pm$. Furthermore, such a map $\tau_E$ is unique up to multiplication by $\{\pm 1\}$-valued functions on $B$, and we have
  \[
    \tau_E^2 = (-1)^{\frac{1}{2}n(n - 1)}.
  \]
\end{proposition}

In our situation, we use this for $E = H^1_\R$ and $2n = 4$. This proposition is a corollary of the following lemma.

\begin{lemma}\label{lem:spinor with anti-lin involution}
  Suppose that a spin structure on $\R^{2n}$ are given and denote the spinor bundles by $S^\pm$. Let $C$ be an anti-linear self-isomorphism of $S^+ \oplus S^-$ associated with the spin structure. (The map $C$ commutes with the Clifford multiplications and $C^2 = (-1)^{\frac{1}{2} n(n + 1)}$.) Then there exists an anti-linear isomorphism
  \[
    \map{\tau}{S^+ \oplus S^-}{S^+ \oplus S^-}
  \]
  that preserves the Real part of the metric on $S^\pm$ and makes the following diagrams commutative.
  \[
    \begin{tikzcd}[column sep=large]
      \R^{2n} \times (S^+ \oplus S^-) \arrow[r, "c"] \arrow[d, "(-\id)\times \tau"] & S^+ \oplus S^- \arrow[d, "\tau"] \\
      \R^{2n} \times (S^+ \oplus S^-) \arrow[r, "c"] & S^+ \oplus S^-
    \end{tikzcd}
  \]
  Here $c$ denotes the Clifford multiplication. If $n$ is even, $\tau$ preserves the grading of $S^\pm$ and if $n$ is odd $\tau$ reverses the grading of $S^\pm$. Furthermore, such a map $\tau$ is unique up to $\{\pm 1\}$, and we have
  \[
    \tau^2 = (-1)^{\frac{1}{2}n(n - 1)}.
  \]
\end{lemma}

\begin{proof}[Proof of \cref{prop:spinor with anti-lin involution}, assuming \cref{lem:spinor with anti-lin involution}]
  Let $P$ be the spin structure on $B$. We have canonical isomorphisms
  \[
    E \cong P \times_{Spin(2n)} \R^{2n},\ S_E^\pm \cong P \times_{Spin(2n)} S^\pm.
  \]
  The commutative diagram in \cref{lem:spinor with anti-lin involution} shows that the following diagram also commutes.
  \[
    \begin{tikzcd}[column sep=large]
      Spin(2n) \times (S^+ \oplus S^-) \arrow[r, "c"] \arrow[d, "\id \times \tau"] & S^+ \oplus S^- \arrow[d, "\tau"] \\
      Spin(2n) \times (S^+ \oplus S^-) \arrow[r, "c"] & S^+ \oplus S^-
    \end{tikzcd}
  \]
  This means that the map $\tau$ induces $\map{\tau_E}{S_E^\pm}{S_E^\pm}$.
  Furthermore, the commutative diagram in \cref{lem:spinor with anti-lin involution} implies that the diagram in \cref{prop:spinor with anti-lin involution} commutes. The uniqueness of $\tau_E$ up to $\{\pm 1\}$-valued functions follows that of $\tau$ up to $\{\pm 1\}$.
\end{proof}

\begin{proof}[Proof of \cref{lem:spinor with anti-lin involution}]
  First we show the uniqueness statement. If $\tau^\prime$ is another map with the above property, then we have the following commutative diagram.
  \[
    \begin{tikzcd}[column sep=large]
      \R^{2n} \times (S^+ \oplus S^-) \arrow[r, "c"] \arrow[d, "\id \times \tau^\prime \circ \tau^{-1}"] & S^+ \oplus S^- \arrow[d, "\tau^\prime \circ \tau^{-1}"] \\
      \R^{2n} \times (S^+ \oplus S^-) \arrow[r, "c"] & S^+ \oplus S^-
    \end{tikzcd}
  \]
  Since $\tau^\prime \circ \tau^{-1}$ is complex linear and preserves the metric, Schur's lemma implies that this is equal to the multiplication by some element in $U(1)$. Furthermore, the commutativity of $\tau^\prime \circ \tau^{-1}$ and $C$ implies that the difference should be $\{\pm 1\}$.

  Next we show the existence statement. By using a volume element $\omega \in Cl(\R^{2n})$, we define
  \[
    \tau = C c(\omega).
  \]
  Here $c(\omega)$ means the Clifford multiplication by $\omega$. This gives the desired isomorphism. Furthermore, a simple calculation shows that $\tau^2 = (-1)^{\frac{1}{2}n(n - 1)}$.
\end{proof}

\bibliographystyle{abbrv}

\begin{thebibliography}{10}

  \bibitem{Adachi-gerbe-like-2026}
  M.~Adachi.
  \newblock A gerbe-like construction in gauge theory.
  \newblock {\em Topology Appl.}, 377:Paper No. 109638, 61, 2026.

  \bibitem{baraglia-mod-2023-arxiv}
  D.~Baraglia.
  \newblock The mod 2 {S}eiberg-{W}itten invariants of spin structures and spin families.
  \newblock {\em arXiv:2303.06883}, 2023.

  \bibitem{Baraglia--Konno-2022}
  D.~Baraglia and H.~Konno.
  \newblock On the {B}auer-{F}uruta and {S}eiberg-{W}itten invariants of families of 4-manifolds.
  \newblock {\em J. Topol.}, 15(2):505--586, 2022.

  \bibitem{Bauer-almost-complex-2008}
  S.~Bauer.
  \newblock Almost complex 4-manifolds with vanishing first {C}hern class.
  \newblock {\em J. Differential Geom.}, 79(1):25--32, 2008.

  \bibitem{Bauer--Furuta-2004}
  S.~Bauer and M.~Furuta.
  \newblock A stable cohomotopy refinement of {S}eiberg-{W}itten invariants. {I}.
  \newblock {\em Invent. Math.}, 155(1):1--19, 2004.

  \bibitem{Furuta-11/8-inequality2001}
  M.~Furuta.
  \newblock Monopole equation and the {$\frac{11}8$}-conjecture.
  \newblock {\em Math. Res. Lett.}, 8(3):279--291, 2001.

  \bibitem{Kronheimer--Mrowka-Dehn-twist-K3-2020}
  P.~B. Kronheimer and T.~S. Mrowka.
  \newblock The {D}ehn twist on a sum of two {$K3$} surfaces.
  \newblock {\em Math. Res. Lett.}, 27(6):1767--1783, 2020.

  \bibitem{Li-quaternionic-bundles-2006}
  T.-J. Li.
  \newblock Quaternionic bundles and {B}etti numbers of symplectic 4-manifolds with {K}odaira dimension zero.
  \newblock {\em Int. Math. Res. Not.}, pages Art. ID 37385, 28, 2006.

  \bibitem{Li--Liu-family-wall-crossing2001}
  T.-J. Li and A.-K. Liu.
  \newblock Family {S}eiberg-{W}itten invariants and wall crossing formulas.
  \newblock {\em Comm. Anal. Geom.}, 9(4):777--823, 2001.

  \bibitem{Lin-bd-Dehn2025}
  Y.~Lin.
  \newblock A note on the boundary {D}ehn twist of ${K}3$ surfaces.
  \newblock {\em arXiv:2506.10444}, 2025.

  \bibitem{morgan1996seiberg}
  J.~W. Morgan.
  \newblock {\em The Seiberg-Witten Equations and Applications to the Topology of Smooth Four-Manifolds}.
  \newblock Princeton University Press, 1996.

  \bibitem{Qiu-Dehn2025-arXiv}
  H.~Qiu.
  \newblock The {D}ehn twist on a connected sum of two homology tori.
  \newblock {\em arXiv:2410.02461}, 2025.

  \bibitem{Ruberman-an-obstruction1998}
  D.~Ruberman.
  \newblock An obstruction to smooth isotopy in dimension {$4$}.
  \newblock {\em Math. Res. Lett.}, 5(6):743--758, 1998.

  \bibitem{Ruberman-polynomial-invarint-of-diffeo-1999}
  D.~Ruberman.
  \newblock A polynomial invariant of diffeomorphisms of 4-manifolds.
  \newblock In {\em Proceedings of the {K}irbyfest ({B}erkeley, {CA}, 1998)}, volume~2 of {\em Geom. Topol. Monogr.}, pages 473--488. Geom. Topol. Publ., Coventry, 1999.

  \bibitem{Ruberman--Strle-mod2-SW-homology-tori-2000}
  D.~Ruberman and S.~s. Strle.
  \newblock Mod 2 {S}eiberg-{W}itten invariants of homology tori.
  \newblock {\em Math. Res. Lett.}, 7(5-6):789--799, 2000.

\end{thebibliography}

\end{document}